\documentclass[10pt, reqno]{amsart}

\usepackage{amsmath,amssymb,amsthm,amsfonts,verbatim}
\usepackage{microtype}
\usepackage[all,2cell]{xy}
\usepackage{mathtools}
\usepackage{graphicx}
\usepackage{pinlabel}
\usepackage{hyperref}
\usepackage{mathrsfs}
\usepackage{color}
\usepackage{enumerate}
\usepackage{cite}
\usepackage{subcaption}
\usepackage{tikz}
\usetikzlibrary{cd}
\tikzcdset{every label/.append style = {font = \small}}

\CompileMatrices

\usepackage[top=1.2in,bottom=1.8in,left=1.2in,right=1.2in]{geometry}

\usepackage{hyperref} %Hyperlink reference
\hypersetup{          %Set up  for hyperlinks
	colorlinks=true, breaklinks, linkcolor=[rgb]{0.15 0 0.8}, filecolor=magenta, urlcolor=blue,
	citecolor=magenta, linktoc=all, }
\usepackage[nameinlink]{cleveref}
\theoremstyle{plain}
\newtheorem{theorem}{Theorem}[section]
\newtheorem{maintheorem}{Theorem}

\newtheorem{maincor}[maintheorem]{Corollary}

\newtheorem{proposition}[theorem]{Proposition}
\newtheorem{lemma}[theorem]{Lemma}

\newtheorem{corollary}[theorem]{Corollary}

\theoremstyle{definition}
\newtheorem{definition}[theorem]{Definition}

\newtheorem{construction}[theorem]{Construction}

\newtheorem{remark}[theorem]{Remark}

\newcommand{\nc}{\newcommand}
\nc{\dmo}{\DeclareMathOperator}

\nc{\Q}{\mathbb{Q}}
\nc{\F}{\mathbb{F}}
\nc{\R}{\mathbb{R}}
\nc{\Z}{\mathbb{Z}}
\nc{\C}{\mathbb{C}}
\nc{\N}{\mathbb{N}}
\nc{\Ell}{\mathcal{L}}
\nc{\M}{\mathcal{M}}
\nc{\K}{\mathcal{K}}
\nc{\I}{\mathcal{I}}
\nc{\T}{\mathcal T}
\nc{\U}{\mathcal U}
\nc{\disk}{\mathbb{D}}
\nc{\hyp}{\mathbb{H}}

\nc{\CP}{\mathbb{CP}}
\nc{\cS}{\mathcal{S}}
\dmo{\Mod}{Mod}
\dmo{\PMod}{PMod}
\dmo{\LMod}{LMod}
\dmo{\Diff}{Diff}
\dmo{\Homeo}{Homeo}
\dmo{\dist}{dist}
\dmo\BDiff{BDiff}
\dmo\SO{SO}
\dmo\Hom{Hom}
\dmo\SL{SL}
\dmo\Sp{Sp}
\dmo\rank{rank}
\dmo\sig{sig}
\dmo\Out{Out}
\dmo\Aut{Aut}
\dmo\Inn{Inn}
\dmo\GL{GL}
\dmo\PSL{PSL}
\dmo\BHomeo{BHomeo}
\dmo\EHomeo{EHomeo}
\dmo\EDiff{EDiff}
\dmo\Disc{Disc}
\nc\Sig{\Sigma}
\dmo\Teich{Teich}
\dmo\Fix{Fix}
\nc{\pair}[1]{\langle #1 \rangle}
\nc{\abs}[1]{\left| #1 \right|}
\nc{\action}{\circlearrowright}
\nc{\norm}[1]{\left | \left | #1 \right | \right |}
\nc{\abcd}[4]{\left(\begin{array}{cc} #1 & #2 \\ #3 & #4 \end{array}\right)}
\nc{\into}\hookrightarrow
\dmo{\Isom}{Isom}
\nc{\normal}{\vartriangleleft}
\dmo{\Vol}{Vol}
\dmo{\im}{Im}
\dmo{\Push}{Push}
\dmo{\Conf}{Conf}
\dmo{\PConf}{PConf}
\dmo{\id}{id}
\dmo{\Jac}{Jac}
\dmo{\Pic}{Pic}
\dmo{\Stab}{Stab}
\dmo{\Arf}{Arf}
\dmo{\End}{End}
\dmo{\Gal}{Gal}
\dmo{\lcm}{lcm}
\dmo{\ab}{ab}
\dmo{\opp}{op}
\dmo{\SU}{SU}
\dmo{\OT}{\Omega \mathcal{T}}
\dmo{\OM}{\Omega \mathcal{M}}
\dmo{\spin}{spin}
\dmo{\even}{even}
\dmo{\odd}{odd}
\dmo{\comp}{\mathcal{H}}
\dmo{\Mgk}{\mathcal{M}_{g, \underline{\kappa}}}
\dmo{\orb}{orb}
\dmo{\AJ}{AJ}
\dmo{\Ck}{\mathsf{C}(\underline{\kappa})}
\dmo{\Int}{Int}
\dmo{\pr}{pr}
\dmo{\lab}{lab}
\dmo{\Sym}{Sym}

\nc{\Span}[1]{\operatorname{Span}(#1)}

\renewcommand{\epsilon}{\varepsilon}
\renewcommand{\tilde}{\widetilde}
\renewcommand{\le}{\leqslant}
\nc{\coloneq}{\mathrel{\mathop:}\mkern-1.2mu=}
\nc{\margin}[1]{\marginpar{\scriptsize #1}}
\nc{\para}[1]{\medskip\noindent\textbf{#1.}}
\nc{\red}[1]{\textcolor{red}{#1}}
\nc{\blue}[1]{\textcolor{blue}{#1}}
\nc{\proofof}[1]{\noindent {\em Proof (of #1).}}

\nc{\lb}{[}
\nc{\rb}{]}

\newcommand{\Si}{\Sigma}

\newcommand{\ld}{\Lambda_{\calD}}
\newcommand{\calD}{\mathcal{D}}

\title{Vanishing cycles, plane curve singularities, and framed mapping class groups}

\author{Pablo Portilla Cuadrado and Nick Salter}
\email{pablo.portilla@cimat.mx}
\email{nks@math.columbia.edu}
\thanks{NS is supported by NSF Award No. DMS-1703181.}
\thanks{PPC is supported by CONACYT project wtih No. 286447.}
\address{PPC: CIMAT, De Jalisco s/n, Valenciana, 36023 Guanajuato, Gto., Mexico }
\address{NS: Department of Mathematics, Columbia University, 2990 Broadway, New York, NY 10027}
\date{May 13, 2020}

\begin{document}
\maketitle	
\begin{abstract}
Let $f$ be an isolated plane curve singularity with Milnor fiber of genus at least $5$. For all such $f$, we give (a) an intrinsic description of the geometric monodromy group that does not invoke the notion of the versal deformation space, and (b) an easy criterion to decide if a given simple closed curve in the Milnor fiber is a vanishing cycle or not. With the lone exception of singularities of type $A_n$ and $D_n$, we find that both are determined completely by a canonical framing of the Milnor fiber induced by the Hamiltonian vector field associated to $f$. As a corollary we answer a question of Sullivan concerning the injectivity of monodromy groups for all singularities having Milnor fiber of genus at least $7$.
\end{abstract}
	
	\section{Introduction}
	
Let $f: \C^2 \to \C$ denote an isolated plane curve singularity and $\Sigma(f)$ the Milnor fiber over some point. A basic principle in singularity theory is to study $f$ by way of its {\em versal deformation space} $V_f \cong \C^\mu$, the parameter space of all deformations of $f$ up to topological equivalence (see \Cref{subsection:versal}). From this point of view, two of the most basic invariants of $f$ are the set of {\em vanishing cycles} and the {\em geometric monodromy group}. A simple closed curve $c \subset \Sigma(f)$ is a {\em vanishing cycle} if there is some deformation $\tilde f$ of $f$ with $\tilde f^{-1}(0)$ a {\em nodal} curve such that $c$ is contracted to a point when transported to $\tilde f^{-1}(0)$. The {\em geometric monodromy group} $\Gamma_f$ is a subgroup of the mapping class group $\Mod(\Sigma(f))$ of the Milnor fiber. Modulo some technicalities, this is defined as the monodromy group of the ``universal family of Milnor fibers'' parameterized by the complement $V_f \setminus \Disc$ of the ``discriminant locus'' in $V_f$. The two notions are connected by the {\em Picard--Lefschetz formula}, which states that the monodromy around a nodal deformation $\tilde f \in \Disc \subset V_f$ is a {\em Dehn twist} $T_c \in \Mod(\Sigma(f))$ around the vanishing cycle $c$ corresponding to that deformation.
	
	In spite of the fundamental importance of these notions, to date, they have not been completely understood. Wajnryb \cite{wajnryb} has investigated the {\em homological} monodromy group of plane curve singularities, arriving at a nearly complete answer in this approximate setting, and using his result to completely describe the {\em homology classes} supporting vanishing cycles. Using the theory of {\em divides}, A'Campo has given an algorithm which builds an explicit model for the Milnor fiber $\Sigma(f)$ equipped with a ``distinguished basis'' of finitely many vanishing cycles. This allows for a description of $\Gamma_f$ as the group generated by the associated Dehn twists, and the set of vanishing cycles as the orbit of any single one under $\Gamma_f$. This is however still not the end of the story, since (a) for a singularity with large Milnor number, A'Campo's algorithm grows impractical, (b) it gives no intrinsic characterization of $\Gamma_f$ as a subgroup of $\Mod(\Sigma(f))$ and (c) it does not address the {\em decidability problem} for vanishing cycles: given some simple closed curve $c \subset \Sigma(f)$,  A'Campo's result gives no method for determining if $c$ is a vanishing cycle if it is not {\em a priori} represented as such. 
	
%	The objective of this paper is to give a complete description of the geometric monodromy group and the vanishing cycles for arbitrary plane curve singularities, addressing each of the above points. We prove that the geometric monodromy group is the stabilizer of a vector field which is canonically associated to $f$ (addressing (a) and (b)), and  we give an easy-to-compute criterion (the vanishing of a simple line integral) that detects if a given simple closed curve appears as a vanishing cycle of a morsification of the singularity (which completely addresses (c)). 
	
	We find that unless $f$ is of type $A_n$ or $D_n$, both $\Gamma_f$ and the set of vanishing cycles are governed completely by a {\em framing} of $\Sigma(f)$. Up to isotopy, a framing is specified by a non-vanishing vector field, and when $f$ is an isolated plane curve singularity, the Hamiltonian vector field $\xi_f$ associated to $f$ is a canonical such choice (see \Cref{subsection:canonicalframing}). In the recent paper \cite{strata3}, A. Calderon and the second author developed the basic theory of {\em framed mapping class groups}. By definition, if $S$ is a surface and $\phi$ is a framing, the framed mapping class group $\Mod(S)[\phi]$ is the {\em stabilizer} of $\phi$ under the action of $\Mod(S)$ on the set of isotopy classes of framings of $S$. See \Cref{section:framings} for an introduction to this theory. Our first main result describes the geometric monodromy group as the full stabilizer of the ``Hamiltonian relative framing'' (a {\em relative} framing is an isotopy class of framing rel. boundary; again see \Cref{section:framings} for details). 
		
	\begin{maintheorem}\label[theorem]{theorem:main}
		Let $f: \C^2 \to \C$ be an isolated plane curve singularity. Suppose that $\Sigma(f)$ has genus $g \ge 5$. If $f$ is not of type $A_n$ or $D_n$, then
		\[
		\Gamma_f = \Mod(\Sigma(f))[\phi],
		\]
		where $\phi$ is the relative framing of $\Sigma(f)$ determined by the Hamiltonian vector field $\xi_f$. 
	\end{maintheorem}
	For an analysis of the $A_n$ and $D_n$ singularities, see \Cref{theorem:ADVC} discussed just below. We remark that the restriction $g \ge 5$ is an artifact inherited from \cite{strata3};  there are exactly six topological types of singularity not addressed by \Cref{theorem:main} or \Cref{theorem:ADVC}. 
	
	As a corollary, we obtain the following precise characterization of which simple closed curves on the Milnor fiber are vanishing cycles. To describe this, we observe that the framing $\phi$ gives rise to a ``winding number function'' on simple closed curves: if $c \subset \Sigma(f)$ is an oriented, $C^1$-embedded simple closed curve, then one measures the winding number of the forward-pointing tangent vector of $c$ relative to $\xi_f$ (see \Cref{subsection:framings}). We say that a simple closed curve $c$ is {\em admissible} for $\phi$ if this winding number is zero.
	
	\begin{maintheorem}\label[theorem]{theorem:VC}
		Let $f: \C^2 \to \C$ be an isolated plane curve singularity. Suppose that $\Sigma(f)$ has genus $g \ge 5$ and that $f$ is not of type $A_n$ or $D_n$. Then a nonseparating simple closed $c \subset \Sigma(f)$ is a vanishing cycle if and only if $c$ is admissible for $\phi$.
	\end{maintheorem}
	
\para{Effectiveness} We emphasize that \Cref{theorem:main,theorem:VC} give effective and simple procedures to determine (a) if a mapping class $f \in \Mod(\Sigma(f))$ arises in the geometric monodromy group and (b) if a simple closed curve $c \subset \Sigma(f)$ is a vanishing cycle. Both of these can be determined from the winding number function $\phi$ mentioned above. The winding number function determines the isotopy class of the framing, and so to check if $f \in \Mod(\Sigma(f))[\phi]$, it suffices to check that $f \cdot \phi = \phi$. This is effectively computable: $\phi$ is in fact nothing more than an element of $H^1(UT\Sigma(f);\Z)$ (here, $UT$ denotes the unit tangent bundle), and so is determined by its values on finitely many curves. Likewise, checking if $c \subset \Sigma(f)$ is a vanishing cycle is extremely simple: one merely computes the winding number, which amounts to expressing $c$ equipped with its forward tangent vector as a cycle in $H_1(UT\Sigma(f);\Z)$, and evaluating $\phi$ on this. From a more analytic point of view, this is equivalent to computing the line integral mentioned above.

\para{Non-injectivity of the monodromy} A question of Sullivan (as recorded by A'Campo \cite{NorGI}) asks when the monodromy representation for an isolated plane curve singularity is {\em faithful}. Work of Perron--Vannier \cite{PV} establishes that this is the case for singularities of type $A$ and $D$, while Wajnryb \cite{Waj2} showed that this is {\em false} for the singularities $E_6,E_7, E_8$ (we should mention here also the work of Labruere \cite{Lab}, which studies a similar non-injectivity phenomenon for ``mapping class group representations'' of abstract Artin groups). Following Wajnryb's work, interest in the injectivity question for general singularities appears to have subsided. As a corollary of \Cref{theorem:main}, we find that {\em only} the singularities of type $A$ and $D$ have injective monodromy representations when the genus of the Milnor fiber is at least $7$.

\begin{maincor}\label{corollary:noninj}
Let $f: \C^2 \to \C$ be an isolated plane curve singularity. Suppose that $\Sigma(f)$ has genus $g \ge 7$ and that $f$ is not of type $A_n$ or $D_n$. Let $F:\C^2 \oplus V_f \to \C$ be a representative of the versal deformation of $f$ where $V_f \subset \C^\mu$ is a small ball in the base space of the versal deformation of $f$. Let $\Disc \subset V_f$ denote the discriminant locus, that is the points $v \in V_f$ such that $0$ is a critical value of $F(\cdot, v)$ in a small ball $B_\epsilon \subset \C^2$. Then the geometric monodromy representation
\[
\rho: V_f \setminus \Disc \to \Mod(\Sigma(f))
\]
is non-injective. 
\end{maincor}
\begin{proof}
By \Cref{theorem:main}, the image of $\rho$ is the framed mapping class group $\Mod(\Sigma(f))[\phi]$. According to \cite[Corollary 3.2]{RW}, $H_1(\Mod(\Sigma(f))[\phi]; \Z) \cong \Z/24\Z$ as long as $\Sigma(f)$ has genus $g \ge 7$. It therefore suffices to show that $\pi_1(V_f \setminus \Disc)$ admits an infinite cyclic quotient. This can be constructed as follows: $\Disc$ is a complex analytic set of complex codimension $1$, so after possibly shrinking $V_f$, we can assume that $\Disc$ is the vanishing locus of a holomorphic function $D: V_f \to \C$. Thus $d \log(D)$ is a holomorphic $1$-form on $V_f \setminus \Disc$ measuring the ``winding number'' around $\Disc$. Integration of $d \log(D)$ provides a quotient $\pi_1(V_f \setminus \Disc) \to 2 \pi i\Z$. Let $\Delta \subset V_f$ be a holomorphic disc intersecting a smooth point of $\Disc$ transversely; then by the residue theorem, $\int_{\partial \Delta}d \log(D) = 2 \pi i$, showing nontriviality.
\end{proof}

\para{The action on relative homology} In \cite{framedMCG}, A. Calderon and the second author determine the action of the framed mapping class group $\Mod(\Sigma(f))[\phi]$ on the {\em relative} homology group $H_1(\Sigma(f), \partial \Sigma(f); \Z)$. Using this result and \Cref{theorem:main}, we can extend the work of Wajnryb \cite{wajnryb} mentioned above and obtain a complete description of the homological monodromy group of a plane curve singularity. In \cite{framedMCG}, the authors construct a {\em crossed homomorphism}

	\[
	\Theta_\phi: \Aut(H_1(\Sigma(f), \partial \Sigma(f);\Z)) \to H^1(\Sigma(f); \Z/2\Z)
	\] 
	that is closely related to the theory of spin structures. By combining \Cref{theorem:main} and \cite[Theorem B]{framedMCG}, we obtain the following corollary.
	
	\begin{corollary}
	Let $f: \C^2 \to \C$ be an isolated plane curve singularity. Suppose that $\Sigma(f)$ has genus $g \ge 5$. If $f$ is not of type $A_n$ or $D_n$, then the ``relative homological monodromy group'' of $f$ is the subgroup $\ker(\Theta_\phi)\le \Aut(H_1(\Sigma(f), \partial \Sigma(f);\Z))$. 
	\end{corollary}
For a more complete discussion of the crossed homomorphism $\Theta_\phi$ and its kernel, we refer the reader to \cite{framedMCG}.

\para{\boldmath Types $A_n$ and $D_n$: hyperellipticity} Up to this point in the discussion we have been excluding the singularities $A_n$ and $D_n$ from consideration. The need for this can be explained topologically as follows. As can readily be seen from the equations, the Milnor fibers for $A_n$ and $D_n$ are {\em hyperelliptic} - they admit involutions $\iota: \Sigma(f) \to \Sigma(f)$ that act as $-1$ on homology. Hyperelliptic phenomena in the theory of mapping class groups always require special consideration, but are usually more tractable. We carry out the necessary analysis in \Cref{section:AD}. Although we suspect that these results are probably known to experts, we include them for the sake of completeness and since they do not require too much effort. We find that again the vanishing cycles are characterized by topological data, but in this case, it is the involution $\iota$, not the framing, which exerts control. To formulate the result, we observe that there is a canonical map of Milnor fibers $p: \Sigma(D_n) \to \Sigma(A_{n-1})$ obtained by capping off one of the boundary components (see \Cref{section:AD}).\\

\noindent \Cref{theorem:ADVC} {\em A non-separating simple closed curve $c \subset \Sigma(A_n)$ is a vanishing cycle if and only if $\iota(c)$ is isotopic to $c$. Likewise, such $c \subset \Sigma(D_n)$ is a vanishing cycle if and only if $p(c)$ is isotopic to $\iota(p(c))$.}\\

\para{Outline of the argument} In light of the Picard--Lefschetz formula, the geometric monodromy group encodes the information of which curves are vanishing cycles; consequently \Cref{theorem:VC} is a straightforward corollary of \Cref{theorem:main}. A first crucial observation in the direction of \Cref{theorem:main} is that the presence of the Hamiltonian vector field $\xi_f$ constrains the geometric monodromy to lie in the associated framed mapping class group. We make the necessary remarks to this effect in \Cref{lemma:framedmonodromy}. Following this, the basic approach to \Cref{theorem:main} is to follow A'Campo's algorithm mentioned above for constructing a model of the Milnor fiber equipped with a distinguished basis of vanishing cycles out of the data of a divide. The theory of framed mapping class groups as developed in \cite{strata3} provides a set of tools to show that a finite collection of Dehn twists (necessarily about {\em admissible} curves) generates the associated framed mapping class group. 

Surfaces with framings that occur ``in nature'' can have a vast diversity of appearances (e.g. as Seifert surfaces for a broad array of knots and links). The work of \cite{strata3} provides generating sets for framed mapping class groups that are effectively ``coordinate free'' in that they do not require one to implicitly rely on an identification of the surface in question with a fixed ``reference'' surface. Rather, \cite{strata3} gives a {\em criterion} for a collection of Dehn twists to generate the framed mapping class group, formulated in terms of the notion of an {\em assemblage} (\Cref{definition:assemblage}). 

An assemblage is an ordered collection of simple closed curves; in particular, it determines a topological filtration $\dots  \subset S_j \subset S_{j+1} \subset \dots$ of the surface by taking regular neighborhoods of increasing subcollections. The {\em core} $S_0$ of an assemblage is required to be built on a subconfiguration $\mathcal C_0$ with some special properties (it must be an ``$E$-arboreal spanning configuration'' of genus $h \ge 5$, c.f. \Cref{definition:config}). Subsequent surfaces $S_{j+1}$ are obtained from $S_j$ by {\em stabilization}: attaching a single $1$-handle. Thus the corresponding curve $c_{j+1}$ is required to enter and exit $S_j$ exactly once, but no further constraints on the intersection of $c_{j+1}$ with curves $c_i (i \le j)$ are imposed. Theorem B of \cite{strata3}, recorded here as \Cref{theorem:Egens}, states that {\em any} assemblage built by repeated stabilization on an $E$-arboreal spanning configuration of genus $h \ge 5$, generates the corresponding framed mapping class group (the framing is characterized by the condition that all curves in the assemblage be admissible).

To prove \Cref{theorem:main}, we must therefore give a description of the Milnor fiber $\Sigma(f)$ as an assemblage of vanishing cycles. Our starting point is A'Campo's model for $\Sigma(f)$ constructed from the data of a {\em divide} of the singularity (c.f. \Cref{section:divides}). This depicts $\Sigma(f)$ as the Seifert surface for a link, and furthermore equips $\Sigma(f)$ with a finite collection $\mathcal C = \{c_1,\dots, c_\mu\}$ of distinguished vanishing cycles. The configuration $\mathcal C$ is not yet an assemblage: there is no specified sequence of attachments, and no core on which $\mathcal C$ restricts to an $E$-arboreal spanning configuration. Solving these two problems forms the heart of the paper.

Construction of the core subsurface takes place in \Cref{section:base,sec:sporadic}. In favorable circumstances (technically, when the multiplicity\footnote{Recall that the {\em multiplicity} of a singularity is defined as the multiplicity of the exceptional divisor appearing in the total transform of the singular fiber under a blowup at the origin. More concretely, the multiplicity is the minimum degree of a monomial appearing in the equation defining $f$.} of the singularity is at least $5$), one can construct a divide for which there is an evident $E$-arboreal spanning configuration consisting entirely of distinguished vanishing cycles. This makes use of the theory of ``divides with ordinary singularities'' as elaborated in Castellini's thesis \cite{Cast}. We discuss this in \Cref{section:base}. The low-multiplicity setting is substantially harder to analyze, and is postponed to the final \Cref{sec:sporadic}. Here, the core subsurfaces are still regular neighborhoods of subsets of distinguished vanishing cycles, but the curves forming the $E$-arboreal spanning configuration are obtained by further manipulating these curves by Dehn twists. 

The remaining problem is to specify a sequence of attachments of the remaining vanishing cycles that is a stabilization (handle attachment) at each step. We treat this in \Cref{section:stabilize}, based on an analysis of the intersection diagram of vanishing cycles associated to the divide. We find that the intersection properties of a distinguished curve relative to the subsurface spanned by some other collection is expressible entirely in terms of the planar graph theory of the intersection diagram. This implies the existence of an ordering on the distinguished vanishing cycles such that the result forms an assemblage on a suitable core. 

As a final complement to these results, we explain how to understand vanishing cycles for the $A_n$ and $D_n$ singularities in \Cref{section:AD}, culminating in \Cref{theorem:ADVC}.

The sections \Cref{section:singularity,section:divides,section:framings} provide the necessary background on singularity theory, the theory of divides, and framed mapping class groups, respectively. 	

\para{Acknowledgements} The second author would like to extend his deep gratitude to Xavier G\'omez Mont for extending an invitation to CIMAT, where this project was begun, for his wonderful hospitality during the visit, and for a very useful discussion of the Hamiltonian vector field of a singularity. Both authors would like to thank CIMAT for hosting a very productive working environment. They are also grateful to Norbert A'Campo for his input on a draft.
	
	\section{Singularity theory}\label{section:singularity}
	
	Let $f:\C^2 \to \C$ be a complex analytic map with an isolated singularity at the origin. That is, $f(0,0) = 0$ and for $\norm{(x,y)}$ suitably small, the partial derivatives of $f$ vanish simultaneously only at $(x,y) = (0,0)$. In this section we collect various results, most of them classical, regarding the theory of plane curve singularities. 
	
	\subsection{Basic notions}
	
	The algebra of convergent power series $\C\{x,y\}$ is a unique factorization domain and so, up to multiplication by a unit, $f$ can be uniquely expressed as $f=f_1\cdot \cdots \cdot f_b$ where each $f_i$ is an irreducible convergent power series. Each $f_i$ is called a {\em branch} of $f$ and if $b=1$ we say that $f$ is {\em irreducible}.
	
	\para{The Milnor fibration} In \cite{Mil}, Milnor proved that for a holomorphic map $f:\C^n \to \C$ with an isolated singularity at the origin, $$\frac{f}{|f|}: S^{2n-1}_\epsilon \setminus K \to S^1$$ is a locally trivial fibration. Here $S_\epsilon^{2n-1} \subset \C^n$ denotes a suitably small sphere and the {\em link} $K$ is defined by $K:= f^{-1}(0) \cap  S_\epsilon^{2n-1}$. If we denote by $D_\delta \subset \C$ a small disk of radius $\delta$ centered at $0$ and by $B_\epsilon \subset \C^n$ a ball of radius $B_\epsilon$, it  follows from Ehresmann's fibration lemma that
	\begin{equation}\label{eq:milnor_fibration}
	f_{|_{f^{-1}(\partial D_\delta) \cap B_\epsilon}}: f^{-1}(\partial D_\delta) \cap B_\epsilon \to \partial D_\delta
	\end{equation} is also a locally trivial fibration for $\epsilon$ small enough and $\delta$ small with respect to $\epsilon$. Moreover, Milnor proved in \cite{Mil} that these fibrations are equivalent. For the matter of the present work, we consider the second fibration, specialized to the case $n = 2$.
		
	Fix a fiber of $f_{|_{f^{-1}(\partial D_\delta) \cap B_\epsilon}}$ over any point of $\partial D_\delta$. We denote this fiber by $\Si(f)$ and we call it {\em the Milnor fiber} of $f$. It is a connected oriented compact surface with non-empty boundary. The Milnor fiber has $b$ boundary components, where $b$ is the number of branches of $f$. Its first Betti number $b_1(\Si(f))$ coincides with $\dim_\C \C\{x,y\}/(\partial f/ \partial x, \partial f/ \partial y)$ and any of these quantities is called {\em the Milnor number of } $f$. This will be denoted $\mu_f$.
	
	\para{Puiseux pairs and intersection multiplicities} There is a vast literature on different collections of complete topological invariants of an isolated plane curve singularity. These are numerical invariants that classify isolated plane curve singularities up to a topological change of coordinates in $\C^2$. We recommend \cite{Bri} as an exhaustive and classical reference on the topic of plane algebraic curves. For the purposes of this paper we choose the Puiseux pairs of each branch of $f$ and the pairwise intersection multiplicities of the branches as our complete set of topological invariants. Next, we give a brief summary of these two.
	
	A finite sequence of pairs of numbers $(p_1,q_1),\ldots,(p_k,q_k)$ is a sequence of {\em essential Puiseux pairs} if and only if 
	\begin{equation}\label{eq:puiseux_ineq}
	\begin{split}
	& 2 \leq p_i < q_i,  \\
	& q_i/(p_1p_2\cdots p_ i) < q_{i+1}/(p_1p_2\cdots p_{i+1}) \text{ and }\\
	&\gcd(q_i, p_1\cdots p_i)= 1
	\end{split}
	\end{equation}
	for all $i =1, \ldots,k$. There is a one-to-one correspondence between sequences of essential Puiseux pairs and topological types of {\em irreducible} plane curve singularities.	
	
	There is a second set of numerical invariants that are more suitable for computation: {\em Newton pairs}. These are a finite sequence of coprime numbers $(p_i, \lambda_i)$ that can be computed from the Puiseux pairs by the recursive formula \begin{equation}\label{eq:newton_pairs}
	\begin{split}
	& \lambda_1:=q_1,  \\
	& \lambda_{i+1}:= q_{i+1} -q_i p_{i+1} + \lambda_i p_{i+1} p_ i 
	\end{split}
	\end{equation}
	
	Newton pairs provide a closed formula for the Milnor number of an irreducible plane curve singularity. Defining $p_{k+1}:=1$, the formula reads
	\begin{equation}\label{eq:milnor_number_irr}
	\mu_f = \sum_{i=1}^k \left( \prod_{j=i}^{k+1}p_{j+1}\right)(p_i-1)(\lambda_i-1).
	\end{equation}
	The above formula follows easily from the construction of the Milnor fiber given in \cite{Nor3}.
	
	Now suppose that $f= f_1\cdots f_b$ is a reducible singularity. Here, more data than the Puiseux or Newton pairs of each branch is required to specify the topological type. One approach is via {\em intersection multiplicities}. For each pair $i,j \in \{1,\ldots,b\}, i \neq j$, we define the {\em intersection multiplicity} between the branches $f_i$ and $f_j$ as $$\nu_{ij}:=\dim_{\C}\C \{x,y\}/(f_i,f_j).$$ This is the usual algebro-geometric definition of local intersection multiplicity, i.e. the dimension as a $\C$ vector space of the local ring at $0$ modulo the ideal generated by $f_i$ and $f_j$.
	
	It follows from \cite[Theorem 10.5]{Mil} and Remark 10.10 therein, that a formula for $\mu_f$ can be given in terms of the Milnor numbers of each branch, the intersection multiplicities and the number of branches. Denoting by $\mu_i$ the Milnor number of the branch $f_i$,  the formula is given by:
	\begin{equation}\label{eq:milnor_total}
	\mu_f = \sum_{i=1}^b\mu_i + 2\sum_{i<j}\nu_{ij} - b + 1.
	\end{equation}
	
	\subsection{The versal deformation space and the geometric monodromy group}\label{subsection:versal}
	\ 
	
	\para{The versal deformation space} We briefly recall here the notion of the {\em versal deformation space} of an isolated singularity; see \cite[Chapter 3]{ArnII} for more details. Recall the algebra 
	\[
	A_f = \C\{x,y\} / (\partial f/ \partial x, \partial f/\partial y);
	\]
	previously we defined the Milnor number of $f$ as the dimension of $A_f$. Let $g_1, \dots, g_\mu \in \C[x,y]$ project to a basis of $A_f$. For $\lambda = (\lambda_1, \dots, \lambda_\mu) \in \C^\mu$, define the function $f_\lambda$ by
	\[
	f_\lambda = f + \sum_{i = 1}^\mu \lambda_i g_i.
	\]
	The {\em versal deformation space} of $f$ is the collection of functions $f_\lambda$ for $\lambda \in \C^\mu$. The {\em discriminant locus} is the subset
	\[
	\mbox{Disc}  = \{\lambda \in \C^\mu \mid f_\lambda^{-1}(0) \mbox{ is not smooth}\}.
	\]
	It can be shown that $\mbox{Disc}$ is an algebraic hypersurface. The discriminant locus is stratified according to the topological type of the singularity of the corresponding curve; the top-dimensional stratum parameterizes curves with a single node. The tautological family 
	\begin{equation}\label{eq:taut}
	X_f = \{(\lambda, (x,y)) \mid (x,y) \in f_\lambda^{-1}(0),\ \lambda \not \in \mbox{Disc}\}
	\end{equation}
	has the structure of a smooth surface bundle with base $V_f \setminus \mbox{Disc}$. Moreover, by intersecting $X_f$ and  $V_f \setminus \mbox{Disc}$ with closed polydisks small enough, we get that the fibers of this fibration are diffeomorphic to the Milnor fiber $\Si(f)$ of \cref{eq:milnor_fibration}. We fix a point in $V_f \setminus \mbox{Disc}$ and we denote, also by $\Si(f)$, the fiber with boundary lying over it.
	
	\para{The geometric monodromy group} To define the geometric monodromy group, we first recall that the {\em mapping class group} of a surface $S$ with boundary is given as
	\[
	\Mod(S) := \pi_0(\Diff^+(S, \partial S)).
	\]
	That is, a mapping class is an isotopy class of diffeomorphisms of $S$, where the isotopies are required to fix the boundary {\em pointwise}. Now let $f:\C^2 \to \C$ be an isolated plane
	curve singularity with Milnor fiber $\Sigma(f)$. 
	
	\begin{definition}\label{definition:geom_monodromy_group}
		The {\em geometric monodromy group} is the image in $\Mod(\Si(f))$ of the monodromy representation of the universal family $X_f$ of \cref{eq:taut}.
	\end{definition}

	\begin{definition}\label{definition:vanishing_cycle}
	A {\em vanishing cycle}  is a simple closed curve $c \subset \Si(f)$ that gets contracted to a point when transported to the nodal curve lying over a smooth point of $\mbox{Disc}$.
\end{definition}

	For any linear form $\ell:\C^2 \to \C$ generic with respect to $f$ and $\epsilon>0$ small enough, the map $$\tilde{f}:=f+\epsilon \ell: \C^2 \to \C$$ only has Morse-type singularities and the corresponding critical values $c_1, \ldots, c_\mu$ are all distinct and close to $0\in \C$. The holomorphic map $\tilde{f}$ is usually called a {\em morsification} of $f$.
	
%	Consider the following piece of data
%	\begin{enumerate}
%		\item  A base point $p_0 \in \partial D_\delta$.
%		\item A collection of $\mu$ pairwise-disjoint disks $D_1, \ldots, D_\mu$ such that each disk contains exactly one critical value $c_k$ of $\tilde{f}$.
%		\item A collection of $\mu$ pairwise-disjoint embedded paths $\gamma_1,\ldots, \gamma_\mu$ such that for each $i\in\{1,\ldots,\mu\}$, $\gamma_i:[0,1] \to D_\delta$ satisfies that  $\gamma_i(0)=p_0$ and $\gamma_i(1) \in \partial D_i$.
%	\end{enumerate} 
%	
%	The existence of Morse functions for any holomorphic map with an isolated critical point, readily implies that the geometric monodromy of $f$ admits
%	a factorization into a product of right-handed Dehn twists. This is because the geometric monodromy of any Morse point is a right-handed Dehn twist
	
	We observe that \Cref{definition:geom_monodromy_group} only depends on the topological type of $f$ since the versal deformation is a topological invariant of the singularity. It can be proven, however, that the geometric monodromy group can be defined from {\em any} morsification $\tilde{f}$ as the image in $\Mod(\Si(f))$ of the fundamental group of the disk punctured at the critical values of $\tilde{f}$. A proof of this fact
	is contained in the proof of \cite[Theorem 3.1, pg. 70]{ArnII}.

	\section{Divides} \label{section:divides}
	
	In this section we introduce the basic notions about divides that we use throughout the text. We can think of a divide as a combinatorial tool to study all topological invariants of a plane curve singularity. In particular it gives a precise description of the topology of the Milnor fiber and of the geometric monodromy of the singularity.
	
	\subsection{The divide of a singularity}
	\begin{definition}\label{def:divide}
		A {\em divide} $\calD$ is the image of a generic relative immersion of a finite number of copies of the interval $[-1,1]$ in a real disk $B$ such that the only singularities are double crossings and the image of the intervals meets the boundary of the disk transversely. The image of each copy is a {\em component} of the divide and when the divide consists of just one copy of the interval, we say that it is {\em irreducible}.
	\end{definition}

	In the setting of singularity theory, divides arise as the real points of well-chosen perturbations of the singularity.

	\begin{definition}\label{definition:adapted_morse}
		Let $\calD \subset D^2$ be a divide and let $f:D^2 \to \R$ be a real smooth map. We say that $f$ is a {\em Morse function adapted to } $\calD$ if it satisfies that:
		\begin{enumerate}
			\item The divide is precisely the zero set of $f$, i.e: $f^{-1}(0) = \calD$.
			\item The map $f$ has its saddle points exactly at the double crossings
			of $\calD$.
			\item Each interior region of $\calD$ contains exactly one maximum or one minimum of $f$.
			\item The intersection of the closure of two distinct interior regions with a maximum consists of an (possibly empty) union of saddle points.
		\end{enumerate}
	\end{definition}
	
	We observe that the above definition does not imply that $f$ itself is a Morse function in the usual sense. Our definition requires that the Morse critical
	points be isolated, but the critical values need not (indeed, typically {\em will not}) be isolated. 
	
		Let $f:\C^2 \to \C$ be an isolated plane curve singularity. It is a classical result (for example, it follows from the theory of Puiseux expansion series) that every plane curve singularity is topologically equivalent to a {\em real singularity} (i.e. one defined by an equation with real coefficients). Moreover, it was proven, independently by A'Campo \cite{NorGI} and Gussein-Zade \cite{Gus} that every plane curve singularity is equisingular to a real singularity with local real branches.  Let $B \subset \R^2 \subset \C^2$ be a real disk containing $0$. It follows from \cite[Theorem 4.4]{ArnII} for a real singularity with local real branches, there exists a morsification $f_s$ which is real, that is, the map $f_s$ is a real Morse map, all the critical points and the critical values of $f_s$ are real and all the saddle points are contained in $f^{-1}(0) \cap B$. Then it follows that $f_s^{-1}(0) \cap B$ is a divide and $f_s$ is a Morse function adapted to it. 
		
	\begin{remark}\label{rem:extended_version}
	One can always get a divide for a topological type of a given real singularity with real branches if one is willing to change the equation to an equation with the same topological type. In certain occasions one might prefer to stay with a given real equation. In these cases one has to allow in \Cref{def:divide} for some copies of the circle to be immersed as well. For the most of this work the reader may assume that we are dealing with divides as in \Cref{def:divide}. In \Cref{sec:sporadic} there is one moment where we allow divides with immersed circles because of a particularly nice way of producing such divides for certain singularities.
	\end{remark}

	\subsection{Divides as planar graphs}
		Divides can be viewed as planar graphs. We introduce here the basic ingredients in this point of view, including the {\em intersection graph} (\Cref{construction:dual}) associated to a divide which will be ubiquitous in what follows. We will employ the graph-theoretic perspective in the inductive portions of the proof of \Cref{theorem:main} (\Cref{section:stabilize}).

	\begin{definition}[Planar graph, planar dual]
		A {\em planar graph} is a continuous map $p: \Gamma \to \C$ of a finite $1$-dimensional CW complex $\Gamma$ that is a homeomorphism onto its image. We further require that $p$ be a $C^1$ embedding on the interior of each $1$-cell of $\Gamma$. A {\em face} of a planar graph is a connected component of $\C \setminus p(\Gamma)$. When the embedding is implicit, we drop reference to it and speak of a planar graph $\Gamma$, faces of $\Gamma$, etc. 
		
		The {\em planar dual} of a planar graph $p: \Gamma \to \C$ is the planar graph constructed as follows: enumerate the faces of $\Gamma$ as $F_1, \dots, F_n$. For each face $F_i$, choose a point $p_i$ in the interior. Whenever {\em distinct} faces $F_i, F_j$ are separated by an edge $e$, join the corresponding $p_i, p_j$ by a $C^1$-embedded segment that crosses the interior of $e$ transversely and intersects no other point of $p(\Gamma)$. 
		
		The {\em bounded planar dual} of $\Gamma$ is obtained from the planar dual by removing the vertex corresponding to the unbounded face and all adjacent edges.
	\end{definition}
	
	\begin{figure}[ht] 
		\labellist
		\endlabellist
		\includegraphics[scale=1]{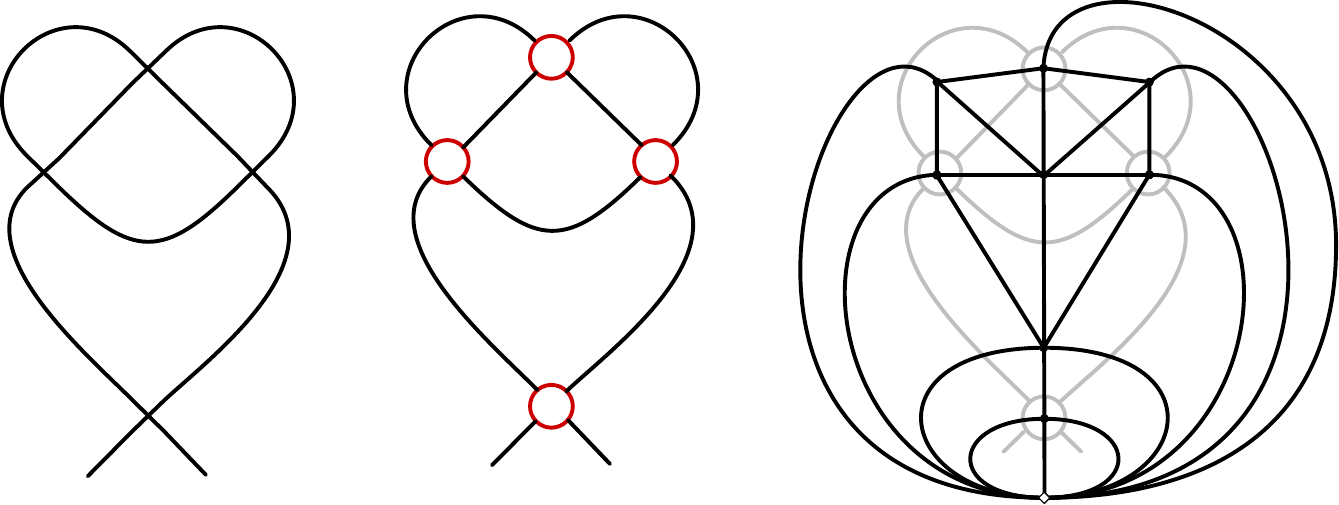}
		\caption{Left to right: the divide $\mathcal D$, the blowup $\tilde{\mathcal D}$, and the augmented intersection graph $\Lambda_{\mathcal D}^+$. The unbounded vertex is the white diamond at the bottom.}
		\label{figure:augmenteddual}
	\end{figure}

	\begin{definition}[(Augmented) intersection graph, (un)bounded vertices]\label{construction:dual}
		Let $\mathcal D$ be a divide. The {\em blowup} $\tilde{\mathcal D}$ of $\mathcal D$ is the planar graph obtained from $\mathcal D$ by replacing each double point in $\mathcal D$ with a small embedded circle. The {\em intersection graph} $\Lambda_{\mathcal D}$ is the bounded planar dual of $\tilde{\mathcal D}$, and the {\em augmented intersection graph} $\Lambda_{\mathcal D}^+$ is the planar dual of $\tilde{\mathcal D}$. We call the vertices of $\Lambda_{\mathcal D} \subset \Lambda_{\mathcal D}^+$ {\em bounded}, and the unique vertex $v_\infty \in \Lambda_{\mathcal D}^+ \setminus \Lambda_{\mathcal D}$ {\em unbounded}. See \Cref{figure:augmenteddual}. 
	\end{definition}
	
	In practice (e.g. in \Cref{figure:divide_o5}), we will not explicitly draw the blowup $\tilde{\mathcal D}$, and instead will pass directly from $\mathcal D$ to $\ld$ by adding a vertex to $\ld$ for each vertex and enclosed region of $\mathcal D$.
	
	The following is a basic property of the augmented divide graph.
	
	\begin{lemma}\label{lemma:dualproperties}
		All bounded faces of $\Lambda_{\mathcal D}^+$ are either triangles or bigons.
	\end{lemma}
	
	\begin{proof}
		By our definition of the planar dual (ignoring self-loops corresponding to self-adjacent faces), the bounded faces of $\Lambda_{\mathcal D}^+$ are in bijection with the non-leaf vertices of $\tilde{\mathcal D}$. For such $v$, if no vertices adjacent to $v \in \tilde{\mathcal D}$ are leaves, then the number of sides in the corresponding face of $\Lambda_{\mathcal D}^+$ is the valence of $v$, which is $3$. If $v \in \tilde{\mathcal D}$ is adjacent to a leaf, then the corresponding face is instead a bigon. See \Cref{figure:augmenteddual}. 
	\end{proof}

	The following lemma recalls some properties that a divide coming from singularity theory satisfies. In \Cref{figure:divide_counterexamples} we see examples of divides that, because of the following lemma, cannot come from singularity theory.
	
	\begin{lemma}\label{lem:properties_divides_singularities}
		Let $f: \C^2 \to \C$ be an isolated plane curve singularity and let $\calD \subset D^2$ be a divide associated to $f$s. The following two properties hold:
		\begin{enumerate}
			\item \label{prop:i} The associated intersection diagram $\ld$ is connected.
			\item \label{prop:iii} The number of intersections points between any two branches of $\calD$ equals the intersection multiplicity of the corresponding two branches of $f$. Hence the number of intersection points between any two branches of $\calD$ is at least $1$.
		\end{enumerate}
	\end{lemma}
	
	\begin{proof}
		The proof of (\ref{prop:i}) is the content of \cite[Section 2]{Gab}. A proof of this fact is also contained in \cite[Theorem 3.4]{ArnII}. Item (\ref{prop:iii})  is in \cite[page 6]{Nor1}.
	\end{proof}

	\begin{figure}[ht!]
	\includegraphics[scale=1]{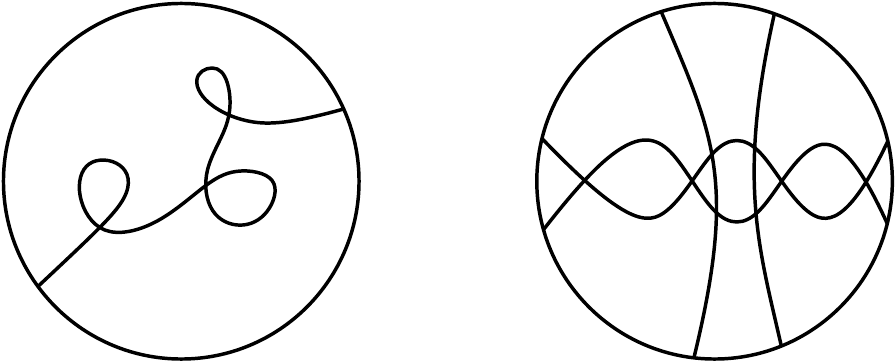}
	\caption{Two examples of divides that cannot come from singularity theory. On the left, a connected divide whose intersection diagram is disconnected. On the right, a connected divide with connected intersection diagram but with two branches not intersecting.}
	\label{figure:divide_counterexamples}
\end{figure}

	\subsection{From the divide to the Milnor fiber}\label{subsection:construction} In \cite{Nor2}, A'Campo shows that the divide can be used to construct a model of the Milnor fiber for which the vanishing cycles distinguished by the divide are easy to understand. We recall the basics of this theory.
	\begin{figure}[ht]
		\labellist
		\small
		\endlabellist
		\includegraphics[scale=1]{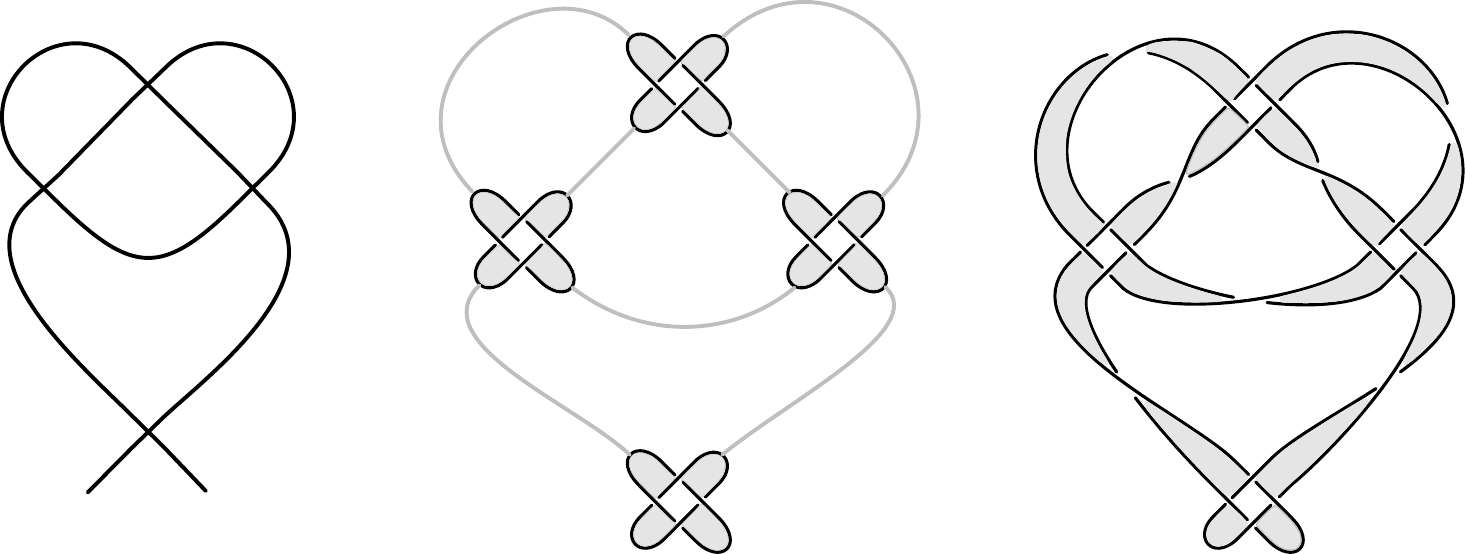}
		\caption{Left: the divide $\mathcal D$. Middle: the first step in the construction of $\Sigma(\mathcal D)$: convert double points in $\mathcal D$ to (boundary-knotted) cylinders. Right: the second and final step in the construction: convert edges in $\mathcal D$ to twisted strips.}
		\label{figure:construction}
	\end{figure}
	
	\begin{figure}[ht]
		\labellist
		\small
		\pinlabel $\Lambda_{\mathcal{D}}^+$ at 20 10
		\pinlabel $\Sigma(\mathcal{D})$ at 210 10
		\endlabellist
		\includegraphics[scale=1]{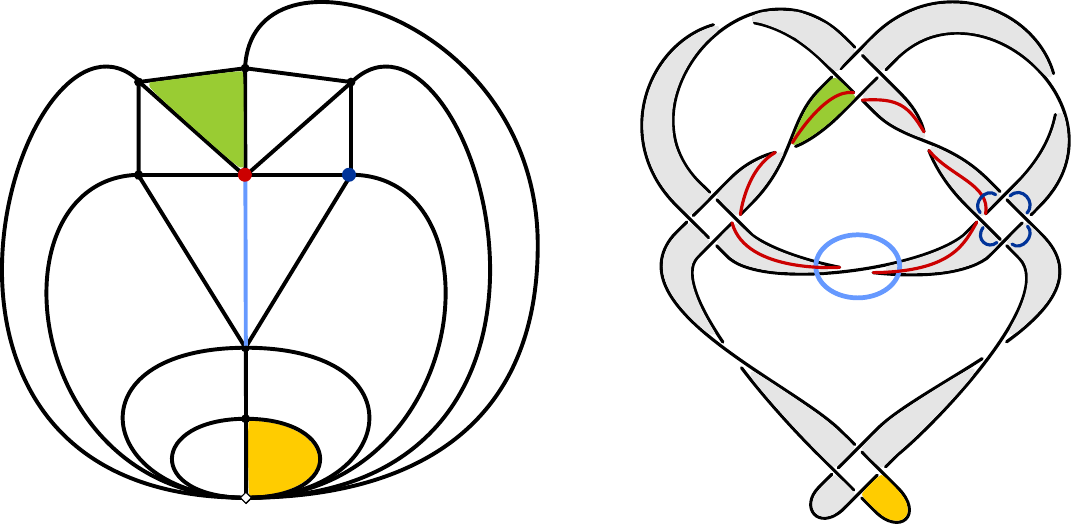}
		\caption{On the left, we have indicated two adjacent vertices, an edge, a triangle, and a bigon of $\Lambda_{\mathcal D}^+$. On the right, we show how these respectively correspond to two intersecting vanishing cycles, a vertical edge, a hexagon, and a square on $\Sigma(\mathcal D)$.}
		\label{figure:legalproof}
	\end{figure}
	
	\begin{construction}[From the divide to the Milnor fiber]\label{construction:milnorfiber}
		Let $\mathcal D$ be a divide. We construct a surface $\Sigma(\mathcal D)$ as follows. Refer to \Cref{figure:construction,figure:legalproof} throughout this discussion. Let $p_1, \dots, p_k$ be an enumeration of the double points in $\mathcal D$. We begin constructing $\Sigma(\mathcal D)$ by taking a disjoint union of cylinders $C_i$ depicted in the middle of \Cref{figure:construction}; note that $C_i$ has four distinguished regions which we identify with the four half-branches of $\mathcal D$ at $p_i$ . For each edge $e \subset \mathcal D$ connecting $p_i$ and $p_j$, we attach the corresponding distinguished regions of the corresponding $C_i$ and $C_j$ by a twisted strip.
		
		Observe that $\Sigma(\mathcal D)$ admits a canonical decomposition into polygonal regions; in fact into squares and hexagons. Each ``half-strip'' corresponding to a half-edge of $\mathcal D$ that connects two saddles is naturally a hexagon with three ``horizontal'' edges determining segments of $\partial(\Sigma(\mathcal D))$ (two lying on the half-strip and one lying on some cylinder), and three ``vertical'' edges lying in the interior of $\Sigma(\mathcal D)$ which connect to other polygonal regions. The half-edges emanating from leaf vertices of $\mathcal D$ instead naturally have the structure of squares with two vertical and two horizontal edges.
		
		The distinguished vanishing cycles for $\mathcal D$ can be identified on $\Sigma(\mathcal D)$ as follows. For $v \in \Lambda_{\mathcal D}$ corresponding to a double point in $\mathcal D$ (i.e. a saddle point), the corresponding vanishing cycle $a_v$ is given by the core of the associated cylinder $C_i \subset \Sigma(\mathcal D)$. If $v \in \Lambda_{\mathcal D}$ instead corresponds to an enclosed region of $\mathcal D$ (i.e. a max or a min), the corresponding $a_v$ is given by following around the strips of $\Sigma(\mathcal D)$ corresponding to the edges of $\mathcal D$ that bound the associated planar region of $\mathcal D$.
	\end{construction}

	For later use, we remark on how the polygonal structure of $\Sigma(\mathcal D)$ is reflected in the augmented intersection graph $\Lambda_{\mathcal D}^+$. The claims follow from an inspection of \Cref{construction:milnorfiber} and \Cref{figure:legalproof}.
	. 
	\begin{lemma}\label{lemma:fiberproperties}
		Let $\mathcal D$ be a divide and $\Sigma(\mathcal D)$ the surface described in \Cref{construction:milnorfiber}. Let $\Lambda_{\mathcal D}^+$ denote the (augmented) intersection graph. For \eqref{item:faces}, recall from \Cref{lemma:dualproperties} that each bounded face of $\Lambda_{\mathcal D}^+$ is either a bigon or a triangle. 
		\begin{enumerate}
			\item\label{item:vertices} The bounded vertices of $\Lambda_{\mathcal D}^+$ are in bijection with the vanishing cycles of $\mathcal D$. Vanishing cycles $a_v, a_w \subset \Sigma(\mathcal D)$ have geometric intersection number $1$ (resp. $0$) if and only if the corresponding vertices $v,w \in \Lambda_{\mathcal D}^+$ are (resp. are not) adjacent. 
			\item The edges of $\Lambda_{\mathcal D}^+$ are in bijection with the vertical edges of $\Sigma(\mathcal D)$. 
			\item\label{item:faces} The bounded faces of $\Lambda_{\mathcal D}^+$ are in bijection with the polygons in $\Sigma(\mathcal D)$. If $F$ has $k$ sides, the corresponding polygon has $k$ vertical edges and $k$ horizontal edges. 
		\end{enumerate}
	\end{lemma}

	\subsection{Subsurfaces supported by vanishing cycles}\label{subsection:subsurfaces} In the body of the argument, we will frequently consider subsurfaces of $\Sigma(\mathcal D)$ determined by subgraphs of $\Lambda_{\mathcal D}$. If $\mathcal C \subset \Lambda_{\mathcal D}$ is such a subgraph, we let
	\begin{equation}\label{subsurface}
	\Sigma(\mathcal C) \subset \Sigma(\mathcal D)
	\end{equation}
	denote a small regular neighborhood of the union of the vanishing cycles $a_v \subset \Sigma(\mathcal D)$ for $v \in \mathcal C$. There is a natural model for such subsurfaces in terms of the polygonal structure on $\Sigma(\mathcal D)$ discussed above. 
	\begin{construction}\label{construction:subsurface}
		Let $\mathcal C \subset \Lambda_{\mathcal D}$ be a subgraph. A curve $a_v \subset \Sigma(\mathcal C)$ associated to $v \in \mathcal C$ can be described as follows: it alternates between {\em crossing through} vertical edges of $\Sigma(\mathcal D)$, and {\em following} horizontal edges. $\Sigma(\mathcal C)$ can be constructed as a suitable regular neighborhood of the union of all vertical and horizontal edges so encountered as $c \in \mathcal C$ varies. If two distinct such curves $a_v, a_w$ cross through the same hexagonal region $F \subset \Sigma(\mathcal D)$, then such a regular neighborhood can be expanded to include all of $F$. See \Cref{figure:legalproof2}. 
	\end{construction}
	
		\begin{figure}[ht]
		\labellist
		\small
		\pinlabel $\Lambda_{\mathcal{D}}^+$ at 20 10
		\pinlabel $\Sigma(f)$ at 210 10
		\pinlabel $\mathcal{C}$ at 40 90
		\pinlabel $v_1$ at 63 137
		\pinlabel $v_2$ at 63 93
		\pinlabel $v_3$ at 63 35
		\endlabellist
		\includegraphics[scale=1]{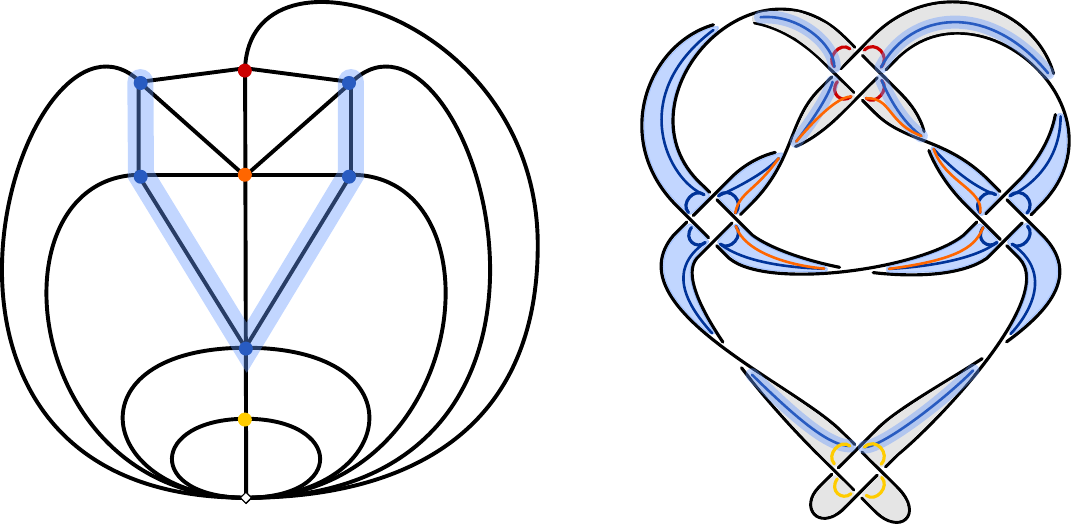}
		\caption{The correspondence between $\mathcal C$ and $\Sigma(\mathcal C)$. The vertices $v_1, v_2, v_3$ {\em not} included in $\mathcal C$ and their corresponding vanishing cycles will be discussed further in \Cref{section:stabilize}.}
		\label{figure:legalproof2}
	\end{figure}
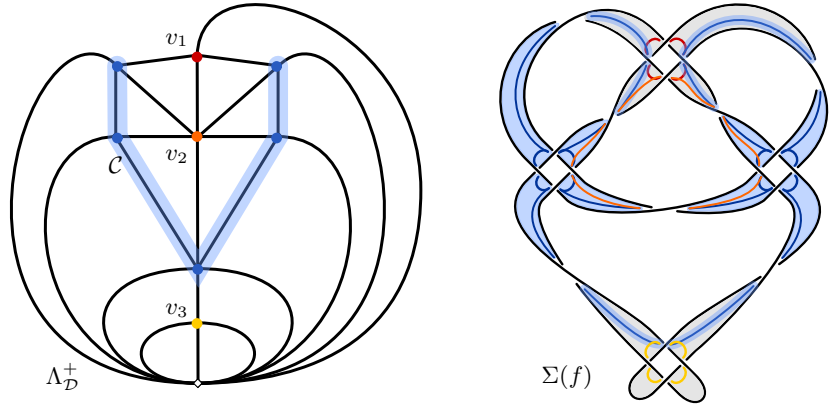

	\section{Framings and the framed mapping class group}\label{section:framings}

	\subsection{Relative framings}\label{subsection:framings}
	We briefly recall here the notion of a {\em relative framing} of a surface; see \cite[Section 2]{strata3} for a more complete discussion. Let $S$ be a compact oriented surface with nonempty boundary. A {\em framing} of $S$ is a trivialization of the tangent bundle of $S$. With a Riemannian metric fixed, framings of $S$ are in correspondence with non-vanishing vector fields on $S$. We say that framings $\phi$ and $\psi$ are {\em isotopic} if the corresponding vector fields are isotopic through non-vanishing vector fields, and are {\em relatively isotopic} if moreover there is such an isotopy that is trivial on $\partial S$. 
	
	\para{(Relative) winding number functions} Suppose that $\phi$ is a framing specified by a non-vanishing vector field $\xi_\phi$. Associated to such data is a {\em winding number function} measuring the holonomy of the forward-pointing tangent vector of simple closed curves on $S$. Precisely, if $\gamma: S^1 \to \mathcal S$ is a $C^1$ embedding, define
	\[
	\phi(\gamma) = \int_{S^1} d\angle (\gamma'(t), \xi_\phi(\gamma(t))) \in \Z.
	\]  
	It is easy to see that $\phi(\gamma)$ is invariant under isotopy of both $\xi_\phi$ and $\gamma$. Letting $\mathcal S$ denote the set of isotopy classes of oriented simple closed curves on $S$, we therefore obtain a function
	\[
	\phi: \mathcal S \to \Z.
	\]
	
	Suppose now that each boundary component $\Delta_i$ of $S$ is equipped with a point $p_i$ such that $\xi_\phi$ is inward-pointing at $p_i$. We call such $p_i$ a {\em legal basepoint}.\footnote{In \cite{strata3}, legal basepoints were required to have $\xi_\phi$ be {\em orthogonally} inward-pointing with respect to some metric, but the treatment we give here is clearly equivalent.} Choose exactly one legal basepoint on each boundary component\footnote{One might be concerned about the possibility that no legal basepoints exist, i.e. that all tangent vectors are outward-pointing. In such a situation, one can modify the framing of the boundary by an isotopy to create legal basepoints. In any case, this only happens when the boundary component has zero winding number, a situation which does not occur for plane curve singularities, c.f. the proof of \Cref{lemma:framedmonodromy}.}. A {\em legal arc} on $S$ is a properly-embedded arc $\alpha: [0,1] \to S$ that begins and ends at distinct legal basepoints, and such that $\alpha$ is tangent to $\xi_\phi$ at both endpoints. The winding number of a legal arc is then necessarily of the form $c + \frac{1}{2}$ for $c \in \Z$, and is invariant up to isotopy through legal arcs ({\em legal isotopies}, for short). Observe also that $\Mod(S)$ acts on the set of legal isotopy classes of legal arcs. 
	
	We let $\mathcal S^+$ be the set obtained from $\mathcal S$ by adding all isotopy classes of oriented legal arcs. Thus, having chosen a system of legal basepoints, a framing $\phi$ gives rise to a {\em relative winding number function}
	\[
	\phi: \mathcal S^+ \to \tfrac{1}{2} \Z.
	\]
	The relative winding number function associated to a framing $\phi$ is clearly invariant under relative isotopies of the framing. Crucially, the converse holds as well. 
	
	\begin{proposition}[c.f. Proposition 2.1, \cite{strata3}]\label{proposition:relwnf}
	Let $S$ be a surface of genus $g \ge 2$, and let $\phi$ and $\psi$ be framings of $S$ that restrict to the same framing of $\partial S$. If the relative winding number functions associated to $\phi$ and $\psi$ are equal, then the framings $\phi$ and $\psi$ are relatively isotopic. 
	\end{proposition}

	\subsection{Framed mapping class groups} We continue with the notation of \Cref{subsection:framings}. We recall that the {\em mapping class group} $\Mod(S)$ is the group of {\em relative} (i.e. boundary-trivial) isotopy classes of orientation-preserving diffeomorphisms of $S$. In particular, there is a well-defined action of $\Mod(S)$ on the set of relative isotopy classes of framings of $S$. For such a framing $\phi$, define the {\em framed mapping class group} to be the stabilizer of $\phi$:
	\[
	\Mod(S)[\phi] = \{f \in \Mod(S) \mid f \cdot \phi = \phi\}.
	\]
	\para{Admissible curves, admissible twists} In the theory of framed mapping class groups, a prominent role is played by the set of Dehn twists that preserve the framing. 
	
	\begin{definition}[Admissible]\label{definition:admissible}
		A simple closed curve $a \subset S$ is {\em admissible} if it is nonseparating and if $\phi(a) = 0$ (necessarily for either choice of orientation). The corresponding Dehn twist $T_a$ preserves $\phi$ (c.f. \cite[Section 2]{strata3}), and is called an {\em admissible twist}. 
	\end{definition}
	\begin{remark}\label{remark:admissible}
	There is a converse to the assertion made in \Cref{definition:admissible} that $T_a$ preserves $\phi$ for $a$ admissible. Namely, let $c \subset S$ be an arbitrary {\em nonseparating} simple closed curve. Then the ``twist-linearity formula'' (c.f. \cite[Lemma 2.4.1]{strata3}) implies that if $c$ is not admissible, there is some $d$ such that $\phi(T_c(d)) \ne \phi(d)$, and hence $T_c$ does not preserve $\phi$.
	\end{remark}

	\para{Generating Mod(S)[$\mathbf{\phi}$]}  When the genus of $S$ is at least $5$, \cite{strata3} establishes that $\Mod(S)[\phi]$ is generated by a finite collection of admissible twists. In order to state the result, we introduce two pieces of terminology: {\em $E$-arboreal spanning configurations} and {\em $h$-assemblages of type $E$}.
	\begin{definition}[$E$-arboreal spanning configuration]\label{definition:config}
	Let $\mathcal C =\{c_1, \dots, c_k\}$ be a collection of curves on $S$, pairwise in minimal position, with the property that each pair of curves intersect in at most one point. Such a configuration determines an intersection graph $\Lambda_{\mathcal C}$. The configuration $\mathcal C$ is said to be {\em arboreal} if $\Lambda_{\mathcal C}$ is a tree, and {\em $E$-arboreal} if moreover $\Lambda_{\mathcal C}$ contains the $E_6$ Dynkin diagram as a subgraph.
	\end{definition}
	For later use, and to illustrate \Cref{definition:config}, we record the particular $E$-arboreal spanning configurations we employ in the paper. We will only encounter configurations with the structure of a ``tripod graph'': that is, a tree with exactly one vertex of valence $3$ and all other vertices of valence $\le 2$. A tripod has three branches; we say that $T$ is {\em the tripod of type $(a,b,c)$} if the branches have lengths $a,b,c$ (here, ``length'' counts the number of vertices of valence $\le 2$; in particular, the tripod of type $(a,b,c)$ has a total of $a+b+c+1$ vertices). 
	
	\begin{lemma}\label{lemma:tripods}
	The tripod graphs of type $(1,4,4), (1,2,6), (2,3,4)$, and $(2,2,5)$ all correspond to surfaces of genus $5$ with one boundary component, and hence are $E$-arboreal spanning configurations on a surface of genus $5$.
	\end{lemma}
	\begin{proof}
	A configuration of curves with ten vertices corresponds to an oriented surface with Euler characteristic $-9$. It is a matter of direct inspection to check that for each of the tripod graphs listed above, this surface has genus $5$. Each also contains the tripod graph of type $(1,2,2)$ (i.e. the $E_6$ graph) as a subgraph.
	\end{proof}

We come now to one of the central definitions of the paper, describing the generating criterion employed in \Cref{theorem:Egens}.
	\begin{definition}[$h$-assemblage of type $E$]\label{definition:assemblage}
	Let $\mathcal C = \{c_1, \dots, c_k, c_{k+1}, \dots, c_\ell\}$ be a collection of curves on $S$. Suppose that (1) $\mathcal C_k= \{c_1, \dots, c_k\}$ is an $E$-arboreal spanning configuration on a subsurface $S_k \subset S$ of genus $h$, and (2) for $j \ge k$, let $S_j$ denote a regular neighborhood of the collection $C_j = \{c_1, \dots, c_j\}$; then $c_{j+1} \cap S_{j}$ is a single arc (possibly, but not necessarily, entering and exiting along the same boundary component of $S_j$), and (3) $S_\ell = S$. The subsurface $S_k$ is called the {\em core} of the assemblage, and the ordering of the curves $c_j$ for $j > k$ is called the {\em attaching sequence}.
	\end{definition}
	\begin{remark}
	Two comments concerning assemblages are in order. First, every $E$-arboreal spanning configuration on $S$ is trivially an $h$-assemblage of type $E$, for $h$ the genus of $S$. Secondly, we remark that assemblages are best understood ``constructively'': one begins with an $E$-arboreal spanning configuration on a subsurface, and then attaches any sequence of topological $1$-handles arising as ``half-curves''. Recall that a surface $S^+$ is said to be a {\em stabilization} of $S \subset S^+$ if $S^+$ is obtained from $S$ by attaching a single $1$-handle, i.e. the neighborhood of an arc in some ambient surface. Thus an $h$-assemblage of type $E$ is obtained from an $E$-arboreal spanning configuration on the genus-$h$ subsurface $S_k$ by subsequently performing any sequence of stabilizations. 
	\end{remark}

	\begin{theorem}[c.f. Theorem B.II of \cite{strata3}]\label{theorem:Egens}
	Let $(S, \phi)$ be a framed surface. Let $\mathcal C = \{c_1,\dots, c_\ell\}$ be an $h$-assemblage of type $E$ on $S$ for some $h \ge 5$. If $\phi(c) = 0$ for all $c \in \mathcal C$, then
	\[
	\Mod(S)[\phi] = \pair{ T_c \mid c \in \mathcal C}.
	\]
	\end{theorem}

	\subsection{The canonical framing} \label{subsection:canonicalframing}
	
	In this section, we show how the Milnor fiber of an isolated plane curve singularity is equipped with a canonical relative framing arising from a ``Hamiltonian vector field''.	
	
	\begin{proposition}\label{prop:framed}
	Let $f: \C^2 \to \C$ be an isolated plane curve singularity. There is a vector field $\xi_f$ on $\C^2$ that is everywhere tangent to the fibers of the Milnor fibration, such that the restriction to a given fiber $\Sigma(f)$ is nowhere-vanishing.
	\end{proposition}
	
	\begin{proof}  Consider the $1$-form $df$. The kernel of this form defines the Milnor fibration, and each fiber is invariant under the flow of the Hamiltonian vector field 
	\[
	\xi_f:=\left( \frac{\partial f}{\partial y}, -\frac{\partial f}{\partial x} \right). 
	\]
	
	Since $f$ is an isolated singularity, $\xi_f$ vanishes only at $(0,0) \in \C^2$. Thus $\xi_f$ determines a nowhere-vanishing vector field on each Milnor fiber $\Sigma(f)$.
	\end{proof}
	
%	Let $\epsilon>>\delta>0$ be as in the Milnor fibration and let  $0<\epsilon'<\epsilon$. Consider the projection map $(B_\epsilon \setminus B_{\epsilon'})\cap f^{-1}(D_\delta) \to D_\delta$. Since $f$ has an isolated singularity at the origin, then the above map is a locally trivial fibration. Since $D_{\delta}$ is contractible, by Ehresmann theorem, the fibration is actually trivial. Therefore, its restriction \[(B_\epsilon \setminus B_{\epsilon'})\cap f^{-1}(\partial D_\delta)\to \partial D_\delta\] is also trivial. Therefore we can cover the total space of the Milnor fibration $f^{-1}(\partial D_{\delta}) \cap B_\epsilon$ by trivializing open sets such that one of these sets is $U_0=(B_\epsilon \setminus B_{\epsilon'})\cap f^{-1}(\partial D_\delta)$ and such that that no other trivializing set intersect the boundaries of the Milnor fibers $\MS_\epsilon \cap f^{-1}(\partial D_\delta)$. This allows to lift the vector field $\partial/\partial \theta$  in $\partial D_\delta$ so that on $U_0$ it preserves the vector field $\xi_f$. As a consequence we get that the framing defined by $\xi_f$ is actually a relative framing.
	
	The presence of the canonical framing has strong consequences for the structure of the geometric monodromy group. %The lemma below is standard; see, e.g. \cite[Lemma 7.10]{strata3}.
	\begin{lemma}\label{lemma:framedmonodromy}
		Let $f: \C^2 \to \C$ be an isolated plane curve singularity. Then the geometric monodromy group $\Gamma_f \le \Mod(\Sigma(f))$ stabilizes the relative isotopy class of the canonical framing $\phi$ associated to $\xi_f$:
		\[
		\Gamma_f \le \Mod(\Sigma(f))[\phi].
		\]
	\end{lemma}
	\begin{proof}
	To simplify notation, throughout this argument we choose $\epsilon >> \delta > 0$ suitably small, and restrict the domain of $f$ to $B_\epsilon \cap f^{-1}(D_\delta)$. In this way, the Milnor fiber is simply the preimage $f^{-1}(x_0)$ for a suitable $x_0 \in D_\delta$. 
	
	Let $\tilde f$ be an arbitrary morsification of $f$ with critical set $X \subset D_\delta \subset \C$; we obtain $\Gamma_f$ as the image of the monodromy map 
	\[
	\rho: \pi_1(D_\delta \setminus X) \to \Mod(\Sigma(f)),
	\]
	where $\Sigma(f) = f^{-1}(x_0)$ is a chosen Milnor fiber. By \Cref{proposition:relwnf}, it suffices to show that $\Gamma_f$ preserves the relative winding number function $\phi$ associated to $\xi_f$.
	
	We begin by seeing that $\Gamma_f$ preserves the value of $\phi$ on simple closed curves. Let $c: S^1 \to \Sigma(f)$ be a $C^1$-embedded simple closed curve, and let $\gamma \subset D_\delta \setminus X$ be a loop based at $x_0$. Then $c$ can be parallel transported around the fibers above $\gamma$; after completing this process we obtain the curve $\rho(\gamma)(c)$, well-defined up to isotopy. For each point $x_t \in \gamma$, the vector field $\xi_f$ endows the fiber $f^{-1}(x_t)$ with an associated {\em absolute} winding number function. The winding number of the parallel transport of $c$ is valued in the discrete set $\Z$ and clearly varies continuously under parallel transport, hence is invariant. This shows that $\phi(c) = \phi(\rho(\gamma)(c))$ for every $c \in \mathcal S$ and every $\gamma \in \pi_1(D_\delta \setminus X)$, proving that $\phi$ is $\Gamma_f$-invariant when restricted to simple closed curves. 
	
	It remains to show that $\phi(\alpha) = \phi(\rho(\gamma)(\alpha))$ for $\alpha$ an arbitrary legal arc on $\Sigma(f)$. In order to argue as we did for simple closed curves, it is necessary to give a construction of parallel transport for {\em legal} arcs. To do so, it suffices to specify a system of legal basepoints on each fiber $f^{-1}(x)$ that varies continuously with $x \in D_\delta \setminus X$. 
	
	To give such a construction, let $\epsilon>>\delta>0$ be as above, and let  $0<\epsilon'<\epsilon$. Consider the projection map $(B_\epsilon \setminus B_{\epsilon'})\cap f^{-1}(D_\delta) \to D_\delta$. Since $f$ has an {\em isolated} singularity at the origin, the above map is a locally trivial fibration, and since $D_{\delta}$ is contractible, the fibration is trivial. We are therefore free to adjust $\xi_f$ by an isotopy so that it is {\em monotone} on each boundary component $\Delta_i(x)$ of each fiber $f^{-1}(x)$. In particular, the locus of legal points on each $\Delta_i(x)$ (i.e. those points where $\xi_f$ is inward-pointing) is a union of $\abs{n_i}$ disjoint open intervals, where $n_i$ is the winding number of $\Delta_i(x)$ viewed as a curve on $f^{-1}(x)$. After adjusting for differing normalization and sign conventions, \cite[pp. 330]{Kha} finds that each $n_i< 0$, so that the locus of legal points is nonempty on each fiber.
	
	To summarize, the above paragraph shows that the locus of legal points across the boundary components $\Delta_i(x)$ forms a (necessarily trivial) fiber bundle over $D_\delta$ with fiber a union of $\abs{n_i}$ disjoint open intervals. It is therefore possible to choose a continuously-varying system of legal basepoints across all fibers of the Milnor fibration. As explained above, this data allows for a notion of parallel transport of {\em legal} arcs, and then it is clear that the winding number of a legal arc is invariant under the action of $\Gamma_f$. We have therefore shown that the relative winding number function $\phi$ associated to the canonical framing $\xi_f$ is invariant under $\Gamma_f$, and by \Cref{proposition:relwnf},  $\Gamma_f$ stabilizes the relative isotopy class of $\xi_f$ as claimed. 
	\end{proof}

	\begin{corollary}\label{corollary:vcadmiss}
	Let $a \subset \Sigma(f)$ be a nonseparating curve that can be represented as a vanishing cycle for some nodal degeneration. Then $a$ is admissible for the Hamiltonian framing $\phi$ of \Cref{prop:framed}, i.e. $\phi(a) = 0$ (necessarily for either choice of orientation of $a$).
	\end{corollary}
	\begin{proof}
	If $a$ is a vanishing cycle, then $T_a \in \Gamma_f$ and hence $T_a$ preserves $\phi$ by \Cref{lemma:framedmonodromy}. By \Cref{remark:admissible}, $a$ is admissible. 
	\end{proof}

	\section{Constructing the core}\label{section:base}
	\Cref{theorem:Egens} provides a criterion for a configuration of admissible curves $\mathcal C \subset (\Sigma(f), \phi)$ to generate the framed mapping class group $\Mod(\Sigma(f))[\phi]$, asserting that it suffices for $\mathcal C$ to be an $h$-assemblage of type $E$ for $h \ge 5$ (recall \Cref{definition:config,definition:assemblage}). To describe an assemblage, one must specify two pieces of data: the {\em core} $S_k \subset \Sigma(f)$, by definition a regular neighborhood of an $E$-arboreal spanning configuration, and the {\em attaching sequence}, the order in which the remaining curves are attached on to the subsurface. In this section we tackle the first of these problems, describing in the ``general case'' (multiplicity at least $5$) how to find an $E$-arboreal spanning configuration of genus $5$. Throughout, we assume that the genus of $\Sigma(f)$ is at least $5$, and that $f$ is not the $A_n$ or $D_n$ singularity (these latter cases are discussed in \Cref{section:AD}). 	

\para{Divides with ordinary singularities} Our approach proceeds by finding a suitable ``universal portion'' of a divide near the origin. This will require us to temporarily relax the definition of divide given above and allow $\calD \subset D^2$ to have all kinds of ``ordinary singularities'' instead of just double points. An {\em ordinary singularity} is one that is topologically equivalent to $x^r-y^r$, that is, a union of $r$ pairwise-transverse lines. The method developed by A'Campo in \cite{NorGI,NorGII} produces divides for plane curve singularities that are generic deformations of divides with ordinary singularities. We use this fact to prove in \Cref{prop:generating_tree_mult_5} that when the multiplicity of a singularity $f$ is at least $5$, then $f$ admits a divide $\mathcal D$ for which there is a subgraph $\mathcal C_0 \le \ld$ that has the structure of an $E$-arboreal spanning configuration on a surface of genus $5$. For lower multiplicity, some more ad-hoc arguments are required; these are deferred to \Cref{sec:sporadic}.

The basic existence result for divides with ordinary singularities is the following.

\begin{lemma}\label{lem:divide_with_ordinary}
	Let $f: \C^2 \to \C$ be a real isolated plane curve singularity of multiplicity $m$. Then $f$ admits a real deformation $\tilde{f}$ such that
	\begin{enumerate}
		\item \label{it:i}The set $f^{-1}(0) \cap D^2 \subset \R^2 \subset \C^2$ is a divide $\tilde\calD$ with ordinary singularities.
		\item \label{it:ii} All branches of $f$ pass through the origin $(0,0) \in \R^2$.
		\item \label{it:iii}In a neighborhood of the origin, $\tilde{f}$ is topologically equivalent to the ordinary singularity $x^{m}-y^{m}$. 
	\end{enumerate}  
\end{lemma}

\begin{proof}
The lemma is a consequence of \cite[Corollary 1.11]{Jong}. For a detailed construction of such divide with ordinary singularities following A'Campo's technique one may look at \cite[Section 2.1]{Cast}.
\end{proof}

\begin{figure}[ht!]
	\includegraphics[scale=1]{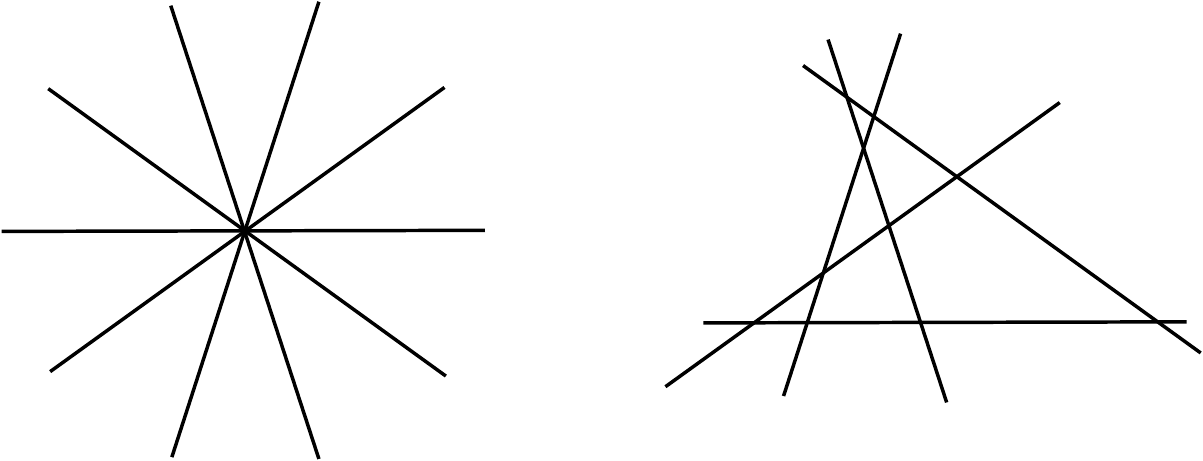}
	\caption{On the left the ordinary singularity $x^5-y^5$. On the right a  divide after a generic deformation.}
	\label{figure:ordinary_deformation}
\end{figure}

Of course, a divide with ordinary singularities can be deformed into a divide in the classical sense. The lemma below asserts that moreover, the deformations at the various singular points can be performed independently of each other. 

\begin{lemma}\label{lem:independence_deformations}
	The process of deforming from a divide with ordinary singularities to a divide can be done at each ordinary singularity in an independent way. That is, we may perform perturbations of the divide near each ordinary singularity assuming that they do not affect the divide outside a small neighborhood of the ordinary singularity.
\end{lemma}

\begin{proof}
	This is the content of \cite[Remark 2.1.19]{Cast} whose proof, according to Castellini, was to communicated to them by A'Campo.
\end{proof}

\begin{proposition}\label{prop:generating_tree_mult_5}
Let $f:\C^2\to \C$ be an isolated singularity with multiplicity $m \geq 5$. Then there exists a divide for $f$ whose intersection diagram contains an $E$-arboreal spanning configuration on $10$ vertices, determining a subsurface of genus $5$.
\end{proposition}

\begin{proof}
	By \Cref{lem:divide_with_ordinary} there exists a divide $\tilde\calD$ with ordinary singularities for $f$, and a point $p \in \tilde\calD$ such that a neighborhood of $p$ looks like the divide of the ordinary singularity $x^{m}+ y^{m}$, that is, a collection of $m$ segments intersecting transversely at $p$.
	
		Let $\xi_1, \ldots, \xi_{2m}$ be the $2m$-th roots of unity. Then the segments of the previous paragraph can be identified with  the union of the $m$ segments $L_i$ that pass through $\xi_i$ and $-\bar{\xi_i} = \xi_{i+m}$ for $i\in\{1, \ldots,m\}$. With this notation $L_1$ corresponds to a horizontal segment lying in the $x$-axis in $\R^2$.
	
	    Take the segments $L_6, \ldots, L_{m}$ and translate them a distance $\epsilon$ along $L_1$ to the negative direction in the $x$ axis. Let $\epsilon'$ be the distance from the union of these segments to $p$.
		
		Now we perturb the segments $L_1, \ldots, L_5$ as in \Cref{figure:divide_o5}. We take the deformation small enough so that all double points emanating from this deformation occur in a small disk of radius $\epsilon'/2$ centered at $p$. Now we perform  perturbations at all remaining ordinary singularities of the divide with ordinary singularities in order to obtain a divide with only double point singularities. Applying \Cref{lem:independence_deformations}, we ensure that after all these perturbations have been made, no new double points appear in the small disk of radius $\epsilon'/2$ centered at $p$.
		
		We have therefore constructed a divide $\mathcal D$ for $f$ such that the red graph of \Cref{figure:divide_o5} is contained as a complete subgraph of the intersection diagram $\ld$. To conclude, we observe that this graph is the tripod graph of type (1,2,6), which is an $E$-arboreal spanning configuration by \Cref{lemma:tripods}.
			 \begin{figure}[ht!]
	\includegraphics[scale=1]{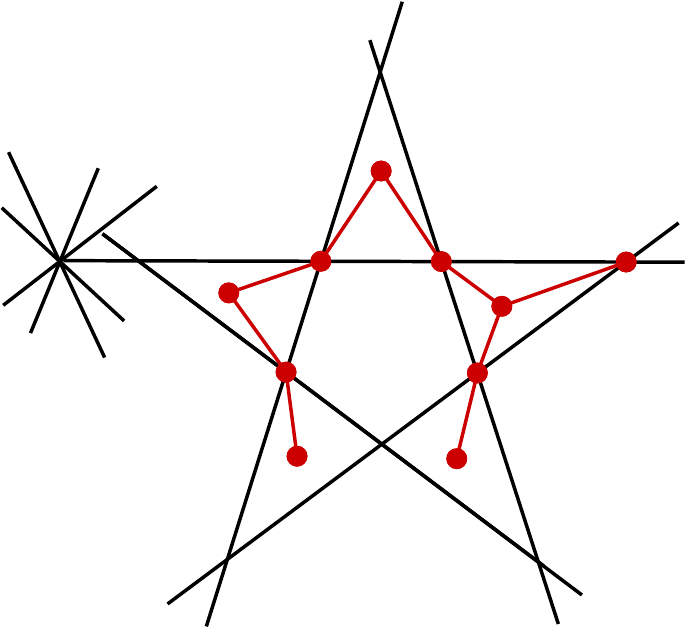}
	\caption{The deformation of $x^m - y^m$ to a divide whose intersection diagram contains a $10$-vertex $E$-arboreal spanning tree (in red). On the left part we see the displaced segments $L_6, \ldots, L_{m}$.}
	\label{figure:divide_o5}
\end{figure}
\end{proof}

We note that \Cref{prop:generating_tree_mult_5} does not address the singularities of multiplicity less or equal than $4$. So as to not break the flow of the paper, and since these cases are analyzed with more ad-hoc methods, we treat them in \Cref{sec:sporadic}. There, we will establish the following counterpart to \Cref{prop:generating_tree_mult_5} (the only substantive difference between these two statements is that in \Cref{prop:generating_tree_mult_5}, the curves in the core are distinguished vanishing cycles for the divide, whereas in \Cref{corollary:lowmult}, the curves are merely in the orbit of such curves under the monodromy action).

	\begin{proposition}\label{corollary:lowmult}
	Let $f:\C^2 \to \C$ be an isolated plane curve singularity of multiplicity $m \le  4$ such that $\Sigma(f)$ has genus at least $5$. If $f$ is not the $A_n$ or $D_n$ singularity, then $f$ admits a divide $\mathcal D$ with a complete subgraph $\mathcal C_0 \le \ld$ such that $\Sigma(\mathcal C_0)$ has genus $5$ and such that there is an $E$-arboreal spanning configuration on $\Sigma(\mathcal C_0)$ consisting of vanishing cycles.
	\end{proposition}

	\section{The attaching sequence} \label{section:stabilize}
	We recall the outline of the proof of \Cref{theorem:main} presented at the start of \Cref{section:base}. Following the work of the previous section (\Cref{prop:generating_tree_mult_5}, \Cref{corollary:lowmult}), we have succeeded in producing the core subsurface $\Sigma(\mathcal C_0) \le \Sigma(\mathcal D)$ of our assemblage. The task in this section is to describe the attaching sequence, i.e. to specify the order in which the subsequent distinguished vanishing cycles are attached to the core via stabilization.
	
	Taken together, \Cref{lemma:legal,lemma:attaching} below show that it is always possible to express $\Sigma(f)$ via a sequence of stabilizations obtained by attaching further vanishing cycles in the divide $\mathcal D$ for $f$ to $\Sigma(\mathcal C_0)$, thereby setting the stage for an application of \Cref{theorem:Egens}.
		
	\subsection{Stabilization and legal attachments}\label{subsection:stabilize1} We discuss here a combinatorial criterion (\Cref{lemma:legal}) under which attaching a vanishing cycle $a_v$ to a subsurface $\Sigma(\mathcal C) \subset \Sigma(\mathcal D)$ yields a stabilization. As \Cref{lemma:legal} shows, only the {\em augmented} dual graph $\Lambda_{\mathcal D}^+$ contains sufficient information, which is why we work in this section with $\Lambda_{\mathcal D}^+$ instead of the perhaps more intuitive $\Lambda_{\mathcal D}$.

	Let $\mathcal P$ be a planar graph and $v \in \mathcal P$ a vertex. The {\em combinatorial tangent space} $T_v \mathcal P$ is the cyclic planar graph whose vertex set consists of the adjacent vertices to $v$, with $w$ and $w'$ adjacent if and only if $w$ and $w'$ are contained in the same face of $\mathcal P$.

	\begin{definition}[Colored, uncolored sets]\label{definition:colored}
		Let $\mathcal C \subset \Lambda_{\mathcal D}^+$ be a subgraph, and let $v \in \Lambda_{\mathcal D}^+ \setminus \mathcal C$ be a vertex. A vertex of $T_v \Lambda_{\mathcal D}^+$ is said to be {\em colored} if the adjacent vertex of $\Lambda_{\mathcal D}^+$ is contained in $\mathcal C$, and {\em uncolored} otherwise. The {\em colored set} $C(\mathcal C, v)$ is the complete subgraph of $T_v \Lambda_{\mathcal D}^+$ spanned by the colored vertices; the {\em uncolored set} $U(\mathcal C, v)$ is defined analogously.
	\end{definition}
	
	\begin{definition}[Legal]\label{definition:legal}
		We say that $v$ is {\em legal rel. $\mathcal C$} if $C(\mathcal C,v)$ is nonempty and connected.
	\end{definition}
	For an illustration of this notion, refer back to \Cref{figure:legalproof2}. There, $C(\mathcal C,v_1)$ has two components, so $v_1$ is not legal rel $\mathcal C$, but $v_2$ and $v_3$ are legal. Observe that correspondingly, $a_{v_1}$ enters and exits $\Sigma(\mathcal C)$ twice, while $a_{v_2}$ and $a_{v_3}$ only do so once. We encode this latter observation in the following lemma.

	\begin{lemma}\label{lemma:legal}
		Let $\mathcal C \subset \Lambda_{\mathcal D}^+$ be a connected subgraph not containing $v_\infty$ and let $\Sigma(\mathcal C) \subset \Sigma(\mathcal D)$ be the corresponding subsurface. Let $v \in \Lambda_{\mathcal D}^+ \setminus \mathcal C$ be a bounded vertex, and let $a_v \subset \Sigma(\mathcal D)$ denote the corresponding vanishing cycle. Then the number of components of $C(\mathcal C, v)$ is equal to the number of components of $a_v \cap \Sigma(\mathcal C)$. In particular, $v$ is legal rel. $\mathcal C$ if and only if $a_v \cap \Sigma(\mathcal D)$ is a single component - if $U(\mathcal C,v)$ is nonempty, $a_v \cap \Sigma(\mathcal D)$ is an arc, and otherwise $a_v$ is contained in $\Sigma(\mathcal D)$.
	\end{lemma}
	
	\begin{proof}
		Consider some $v \in \Lambda_{\mathcal D}^+ \setminus \mathcal C$ and the corresponding vanishing cycle $a_v \subset \Sigma(\mathcal D)$. As $a_v$ runs through $\Sigma(\mathcal D)$, it visits the polygonal regions corresponding to the faces of $\Lambda_{\mathcal D}^+$ incident to $v$, doing so in the order prescribed by the planar embedding of $\Lambda_{\mathcal D}^+$. Given the standard model for $\Sigma(\mathcal C)$ described in \Cref{construction:subsurface}, it is easy to see that the number of times that $a_v$ enters $\Sigma(\mathcal C)$ is equal to the number of times that the corresponding point in $T_v \Lambda_{\mathcal D}^+$ changes from being uncolored to colored, so in total, the number of components of $a_v \cap \Sigma(\mathcal C)$ is equal to the number of colored regions in $T_v \Lambda_{\mathcal D}^+$ as claimed. 
	\end{proof}

	\subsection{Existence of a legal attachment}\label{subsection:stabilize2}
	\Cref{lemma:legal} gives a graph-theoretic description of when a vanishing cycle can be added to a subsurface as a stabilization. We show in \Cref{lemma:attaching} that such a ``legal'' vertex always exists.
	
	\begin{figure}[ht]
		\labellist
		\tiny
		\pinlabel $w_1$ at 30 120
		\pinlabel $w_2$ at 40 100
		\pinlabel $v_1$ at 10 105
		\pinlabel $v_N$ at 60 130
		\small
		\pinlabel $\mathcal{R}$ at 100 50
		\pinlabel $\mathcal{W}$ at 300 45
		\pinlabel $(A)$ at 100 0
		\pinlabel $(B)$ at 325 0
		\endlabellist
		\includegraphics[scale=1]{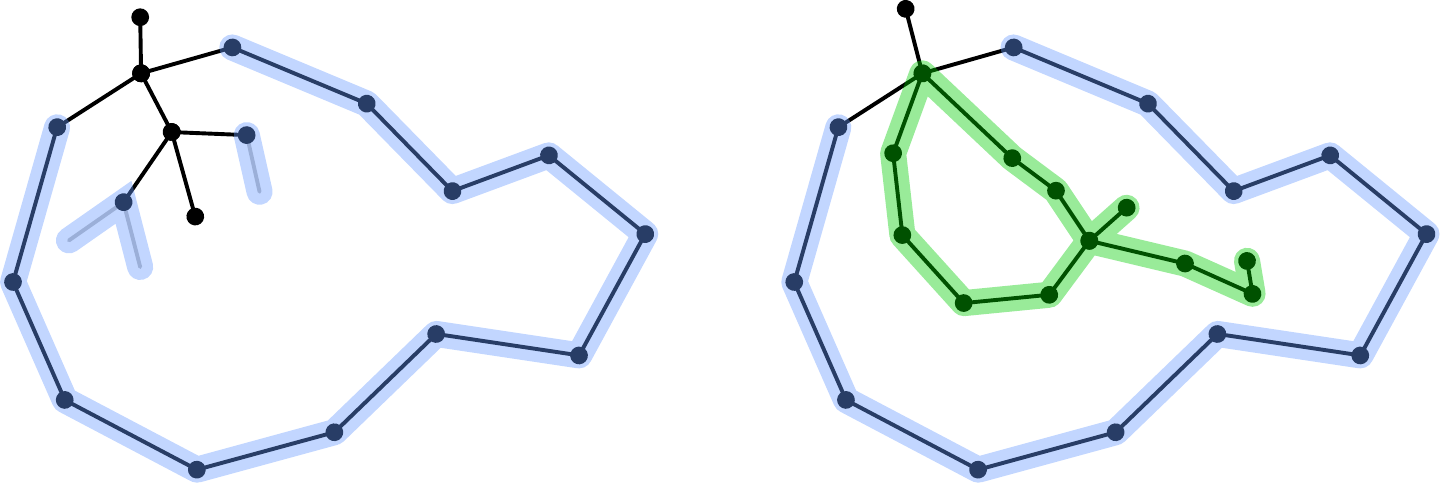}
		\caption{(A): The construction of the bounded region $\mathcal R$. The vertices of $\mathcal C$ are highlighted in blue. (B): The situation where $w_2$ is not adjacent to any elements of $\mathcal C$. At a minimum, $\mathcal W$ could consist just of a single edge between $w_1$ and $w_2$, and $v_1$ and $v_n$ could be joined by a single edge, but the face $F$ would have five sides in this case.}
		\label{figure:attaching}
	\end{figure}
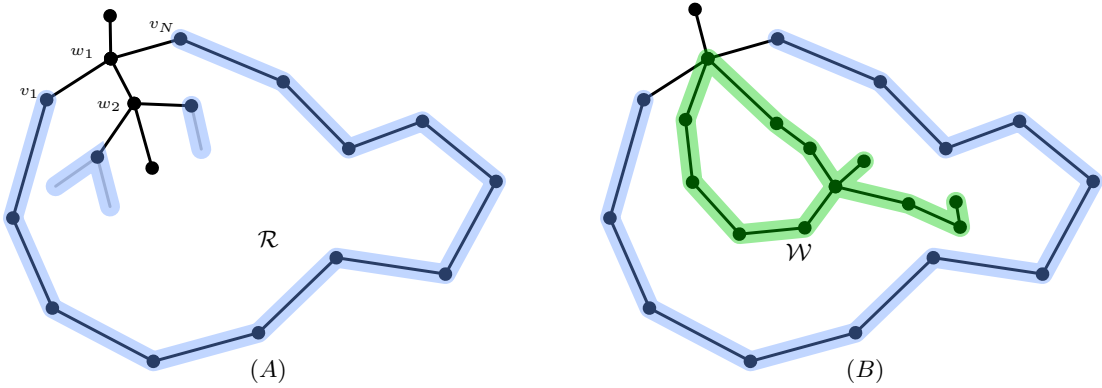

	\begin{lemma}\label{lemma:attaching}
		Let $\mathcal D$ be a divide, and let $\mathcal C \subset \Lambda_{\mathcal D}^+$ be a proper connected complete subgraph. Then there is some bounded vertex $v \in \Lambda_{\mathcal D}^+ \setminus \mathcal C$ such that $v$ is legal with respect to $\mathcal C$ in $\Lambda_{\mathcal D}^+$. 
	\end{lemma}
	
	\begin{proof}
		We begin by emphasizing that we must produce a {\em bounded} vertex, i.e. a vertex of $\Lambda_{\mathcal D}$, but the legality we consider is with respect to the larger graph $\Lambda_{\mathcal D}^+$. In the argument below we will work exclusively with $\Lambda_{\mathcal D}^+$, and so we must take care that the vertex we identify does not correspond to the unbounded face. We will make note of this below at the appropriate stage of the argument. 
		
		Suppose that $w_1$ is a bounded vertex of $\Lambda_{\mathcal D}^+ \setminus \mathcal C$ that is adjacent to some vertex of $\mathcal C$. If $w_1$ is legal rel. $\mathcal C$, there is nothing to show. Suppose then that $v_1, v_N \in \mathcal C$ are adjacent to $w_1$ and lie in adjacent but distinct components of the colored set $C(\mathcal C, w_1)$. Since $\mathcal C$ is connected, there is a path $v_1, \dots, v_N$ of vertices of $\mathcal C$ connecting $v_1$ to $v_N$. Along with the edges joining $v_1$ and $v_N$ to $w_1$, this forms a closed loop in the plane which therefore bounds a closed planar region $\mathcal R$ containing a finite number of vertices of $\Lambda_{\mathcal D}^+$. By abuse of notation we will use $\mathcal R$ to refer both to the planar region and the set of vertices of $\ld^+$ contained in $\mathcal R$.
		
		We observe that the region $\mathcal R$ necessarily does not contain the unbounded vertex $v_\infty$. If $v_\infty \in \mathcal C$ this is clear, so we assume that $v_\infty \not \in \mathcal C$. Observe that $v_\infty$ is adjacent to the unbounded face of $\Lambda_{\mathcal D}^+$ and hence lies in the unbounded component determined by any circuit in $\Lambda_{\mathcal D}^+$ that does not pass through $v_\infty$. Thus $v_\infty$ is not contained in the bounded region $\mathcal R$. The arguments to follow will show that some (necessarily bounded) vertex of $\mathcal R$ is legal rel. $\mathcal C$, and accordingly the remainder of the argument discusses only bounded vertices; we will not mention this every time.

		Since $v_1$ and $v_N$ lie in different components of $C(\mathcal C,{w_1})$, there is some $w_2 \in \mathcal R \setminus \mathcal C$. If $w_2$ is legal rel. $\mathcal C$ the argument concludes. If $w_2$ is not legal, there are two possibilities: either $w_2$ is adjacent to no vertices of $\mathcal C$, or else $C(\mathcal C, w_2)$ has at least two components. 
		
		Suppose first that $w_2$ is adjacent to no vertices of $\mathcal C$. We consider the subgraph $\mathcal W \subset \Lambda_{\mathcal D}^+$ defined as follows: a vertex $v$ is in $\mathcal W$ if and only if $v \in \mathcal R \setminus \mathcal C$ and there is some path of vertices in $\mathcal R$ connecting $v$ to $w_1$. We then take $\mathcal W$ to be the complete subgraph on this vertex set. Note that by construction, $\mathcal W$ is a connected planar graph contained in the planar region $\mathcal R$. 
		
		We claim that some $w_3 \in \mathcal W \setminus \{w_1\}$ must be adjacent to $\mathcal C$. If not, we observe that the exterior of $\mathcal W$ is bounded entirely by a single face $F$ of $\Lambda_{\mathcal D}^+$. Such $F$ must have at least five sides (see \Cref{figure:attaching}(B)). But $F$ is a bounded face of $\Lambda_{\mathcal D}^+$, and hence has at most three sides by \Cref{lemma:dualproperties}. Thus we can repeat the above argument with $w_3$ in place of $w_1$; note that the planar region $\mathcal R'$ produced can be taken to be a strict subset of $\mathcal R$, and so this process can be repeated only finitely many times. 
		
		In the latter case where $C(\mathcal C,{w_2})$ has at least two components, we can repeat the above arguments with $w_2$ in place of $w_1$. Observe that as in the previous case, the new planar region $\mathcal R'$ produced can be taken to be a strict subset of $\mathcal R$. Thus this process must terminate after finitely many steps, showing the existence of $v$ as claimed. 
	\end{proof}

	\subsection{Proof of \Cref{theorem:main}}\label{subsection:proof} Let $f: \C^2 \to \C$ be an isolated plane curve singularity with $\Sigma(f)$ of genus at least $5$. According to \Cref{theorem:Egens}, we must describe an $h$-assemblage of type $E$ on $\Sigma(f)$. Thus we will specify a core subsurface $\Sigma(\mathcal C_0)$ of genus $5$ equipped with an $E$-arboreal spanning configuration of vanishing cycles, and subsequently we will give an attaching sequence of vanishing cycles. 
	
	If $f$ is not the $A_n$ or $D_n$ singularity, then by \Cref{prop:generating_tree_mult_5} or \Cref{corollary:lowmult} as appropriate, there exists a divide $\mathcal D$ for $f$ and a connected subgraph $\mathcal C_0 \le \ld$ such that $\Sigma(\mathcal C_0)$ has genus $5$ and such that $\Sigma(\mathcal C_0)$ contains an $E$-arboreal spanning configuration of vanishing cycles. We take such $\Sigma(\mathcal C_0)$ as the core.
		
	We now describe the attaching sequence of the assemblage. Form the augmented intersection graph $\Lambda_{\mathcal D}^+$ of $\mathcal D$ as in \Cref{construction:dual}. By \Cref{lemma:fiberproperties}.\ref{item:vertices}, the bounded vertices of $\Lambda_{\mathcal D}^+$ correspond to the vanishing cycles of $\mathcal D$. Color the vertices of $\mathcal C_0 \subset \Lambda_{\mathcal D}^+$. If $\mathcal C_0 = \Lambda_{\mathcal D}$, then $\Sigma(\mathcal C_0) = \Sigma(f)$ and \Cref{theorem:main} is proved. Otherwise, apply \Cref{lemma:attaching} to produce a legal vertex $v \in \Lambda_{\mathcal D}$ rel. $\mathcal C_0$, corresponding to a vanishing cycle $a_v$. 
	
	By \Cref{lemma:legal}, since $v$ is legal rel $\mathcal C_0$, the intersection $a_v \cap \Sigma(\mathcal C_0)$ is either all of $a_v$ or else a single arc. In the former case, add $v$ to the set of colored vertices, but do not add $a_v$ to the attaching sequence (topologically, there is no effect to attaching a curve already contained in the subsurface). Otherwise, add $v$ to the colored vertices and add $a_v$ to the attaching sequence.  We now repeat this argument until the uncolored vertices are exhausted, adding the corresponding curves to the attaching sequence whenever they are not already contained in the current subsurface. 
	
	The result of this procedure gives an attaching sequence for a $5$-assemblage of type $E$ on $\Sigma(f)$; each constituent curve $c$ is a vanishing cycle, and hence the corresponding $T_c \in \Gamma_f$. By \Cref{theorem:Egens}, 
	\[
	\Mod(\Sigma(f))[\phi] \le \Gamma_f.
	\]
	As also $\Gamma_f \le \Mod(\Sigma(f))[\phi]$ by \Cref{lemma:framedmonodromy}, \Cref{theorem:main} follows.  \qed
	
	\subsection{Proof of \Cref{theorem:VC}}Let $f: \C^2 \to \C$ be an isolated plane curve singularity and assume that its Milnor fiber $\Si(f)$ has genus greater or equal than $5$ and that $f$ is not $A_n$ or $D_n$. 
	
	Let $a \subset \Sigma(f)$ be a nonseparating simple closed curve such that the associated Dehn twist $T_a$ is in  $\Gamma_f$. By \Cref{corollary:vcadmiss}, $a$ must be an admissible curve for the canonical framing $\phi$ of $\Sigma(f)$. Conversely, if $a$ is an admissible curve, then $T_a \in \Mod(\Sigma(f))[\phi]$. By \Cref{theorem:main}, there is an equality $\Gamma_f = \Mod(\Si(f))[\phi]$ and hence $T_a \in \Gamma_f$.
	 We claim that if $T_a \in \Gamma_f$, then necessarily $a$ is a vanishing cycle. To see this, choose any geometric vanishing cycle $b \subset \Si(f)$ associated to a distinguished basis associated with a morsification of $f$. This is, again, an admissible curve. By \cite[Lemma 7.5]{NickToric} there exists a collection of admissible curves $a=a_1, \ldots, a_k=b$ such that $i(a_i,a_{i+1})=1$ for all $i=1, \ldots, k-1$. Then the automorphism \[\psi:=\prod_{i=1}^{k-1} T_{a_i}T_{a_{i+1}}\] satisfies $\psi(a)=b$ and moreover $\psi \in \Mod(\Si(f))[\phi]= \Gamma_f$.  It is a classical theorem that the set of geometric vanishing cycles forms an orbit of the geometric monodromy group (c.f. \cite[Theorem 3.4]{ArnII}), so $a$ is also a vanishing cycle.	\qed

\section{Types $A$ and $D$}\label{section:AD}
	
	In this section we analyze the exceptions to \Cref{theorem:main}: the singularities $A_n$ and $D_n$. We will see that there is a strong obstruction for \Cref{theorem:main} to apply: these singularities are ``hyperelliptic'', a property which makes its presence felt in both topological and algebro-geometric ways. However, in many respects, the hyperelliptic setting is more tractable, and we are able to obtain the counterpart to \Cref{theorem:VC} (\Cref{theorem:ADVC}) in an essentially straightforward way.

	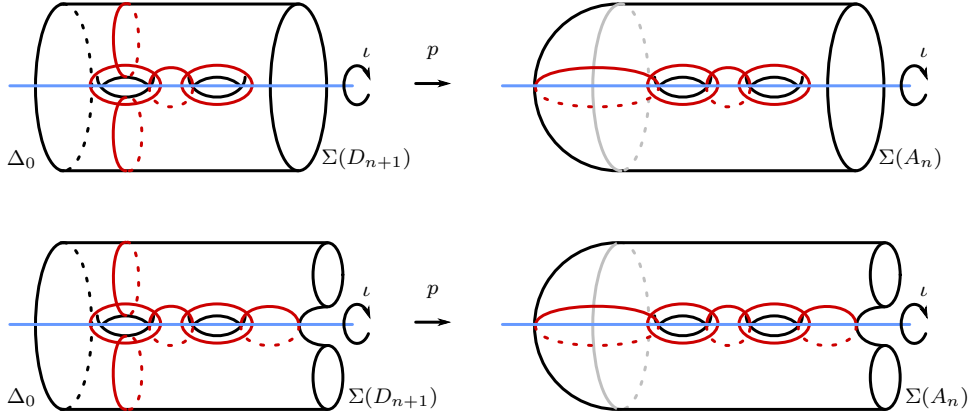
\begin{figure}[ht]
		\labellist
		\Small
		\pinlabel $\Delta_0$ at 5 5
		\pinlabel $\Delta_0$ at 5 95
		\pinlabel $\Sigma(D_{n+1})$ at 135 95
		\pinlabel $\Sigma(A_n)$ at 340 95
		\pinlabel $\Sigma(D_{n+1})$ at 145 5
		\pinlabel $\Sigma(A_n)$ at 350 5
		\pinlabel $\iota$ at 135 45
		\pinlabel $\iota$ at 345 45
		\pinlabel $\iota$ at 135 135
		\pinlabel $\iota$ at 345 135
		\pinlabel $p$ at 160 45
		\pinlabel $p$ at 160 135
		\endlabellist
		\includegraphics[scale=1]{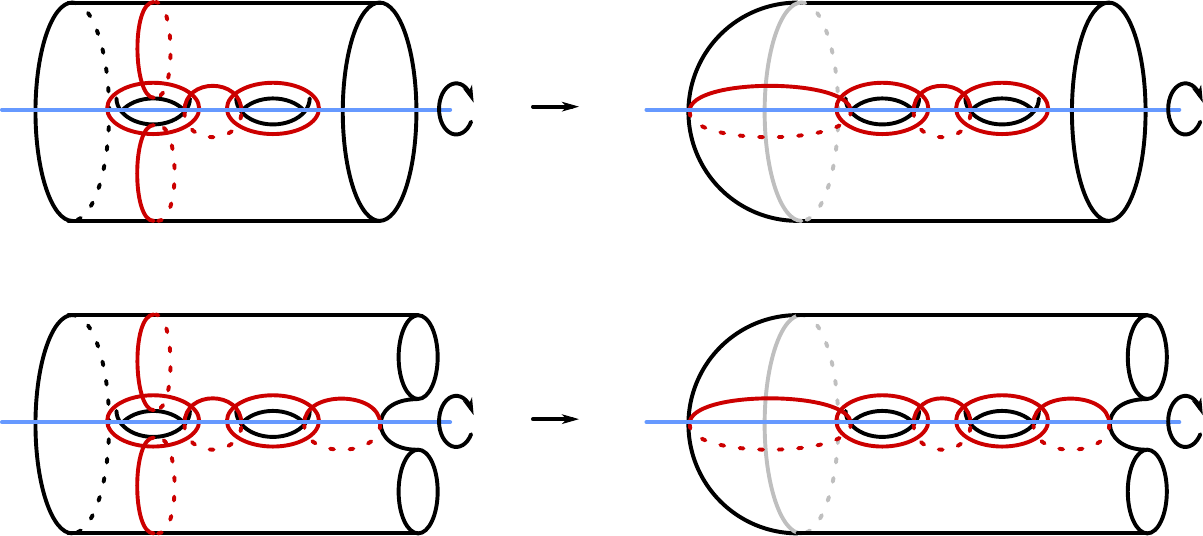}
		\caption{The Milnor fibers $\Sigma(A_n)$ (at right) and $\Sigma(D_{n+1})$ (at left), equipped with the usual distinguished collections of vanishing cycles. Top and bottom represent different parity regimes. As indicated, the Milnor fibers $\Sigma(D_{n+1})$ and $\Sigma(A_n)$ are related by the operation of capping the boundary component $\Delta_0$. We have also illustrated the hyperelliptic involution $\iota$ about which the vanishing cycles are (pre-) symmetric.}
		\label{figure:AnDn}
	\end{figure}

\Cref{figure:AnDn} shows models for the Milnor fibers $\Sigma(A_n), \Sigma(D_n)$ of the $A_n$ and $D_n$ singularities along with the standard distinguished bases of vanishing cycles. As indicated in the figure, on both $\Sigma(A_n)$ and $\Sigma(D_n)$, there is an involution $\iota$ that is ``hyperelliptic'' in the sense that the induced map on homology is the map $x \mapsto - x$. The existence of such an involution can be understood algebro-geometrically as follows: the equations $y^2 - x^{n-1}$ and $x(y^2 - x^{n-2})$ defining $A_n$ and $D_n$ each admit the obvious symmetry $y \mapsto -y$.

In the case of $A_n$, each of the distinguished vanishing cycles is invariant under $\iota$, and hence the entire monodromy group $\Gamma_{A_n}$ is centralized by $\iota$. In the case of $D_n$, the involution $\iota$ does not directly fix each vanishing cycle, but as shown in \Cref{figure:AnDn}, after capping the boundary component $\Delta_0$, the vanishing cycles for $D_n$ become symmetric. We formalize this as follows. 

\begin{definition}[(Pre-)symmetric]
A non-separating simple closed curve $c \subset \Sigma(A_n)$ is {\em symmetric} if it is invariant (up to isotopy) under the topological hyperelliptic involution $\iota: \Sigma(A_n) \to \Sigma(A_n)$ shown in \Cref{figure:AnDn}. A non-separating simple closed curve $c \subset \Sigma(D_n)$ is {\em pre-symmetric} if the image $p(c) \subset \Sigma(A_{n-1})$ is symmetric. 
\end{definition}

\begin{theorem}\label{theorem:ADVC}
A non-separating simple closed curve $c \subset \Sigma(A_n)$ is a vanishing cycle if and only if it is symmetric. Likewise, such $c \subset \Sigma(D_n)$ is a vanishing cycle if and only if it is pre-symmetric. 
\end{theorem}	

\begin{proof}
The assertion in the $A_n$ case is classical: the set of vanishing cycles is the orbit of a single vanishing cycle under the action of the geometric monodromy group $\Gamma_{A_n} \le \Mod(\Sigma(A_n))$ (c.f. \cite[Theorem 3.4]{ArnII}). This latter group is the well-known {\em hyperelliptic mapping class group}; in particular, all symmetric curves are in the same orbit of $\Gamma_{A_n}$, and any given vanishing cycle is easily seen to be symmetric. 

The case of $D_n$ is similarly easy to analyze. The boundary-capping map $p: \Sigma(D_n) \to \Sigma(A_{n-1})$ induces the Birman exact sequence
\[
1 \to \pi_1(UT\Sigma(A_{n-1})) \to \Mod(\Sigma(D_n)) \to \Mod(\Sigma(A_{n-1})) \to 1,
\]
where $UT\Sigma(A_{n-1})$ is the unit tangent bundle of $\Sigma(A_{n-1})$ and $\pi_1(UT\Sigma(A_{n-1}))$ acts as the ``disk-pushing subgroup'' of $\Sigma(D_n)$ about the distinguished boundary component $\Delta_0$ (c.f. \cite[Section 4.2.5]{Farb}). Restricted to the geometric monodromy groups, this yields
\[
1 \to \pi_1(UT\Sigma(A_{n-1}))\cap \Gamma_{D_n} \to \Gamma_{D_n} \to \Gamma_{A_{n-1}} \to 1.
\]
The assertion of \Cref{theorem:ADVC} is therefore equivalent to seeing that $\pi_1(UT\Sigma(A_{n-1})) \cap \Gamma_{D_n}$ acts {\em transitively} on the simple closed curves in $\Sigma(D_n)$ sitting above a fixed isotopy class in $\Sigma(A_{n-1})$. 

Let $\overline{\Sigma(D_n)}$ denote the surface obtained from $\Sigma(D_n)$ by replacing the boundary component $\Delta_0$ with a puncture. Observe that there is a canonical bijection between the isotopy classes of curves on $\Sigma(D_n)$ and on $\overline{\Sigma(D_n)}$. It therefore suffices to consider the image $\overline{\Gamma(D_n)} \le \Mod(\overline{\Sigma(D_n)})$ and its action on curves on $\overline{\Sigma(D_n)}$. 

We claim that the image of $\pi_1(UT\Sigma(A_{n-1})) \cap \Gamma_{D_n}$ in $\Mod(\overline{\Sigma(D_n)})$ is the entire point-pushing subgroup $\pi_1(\Sigma(A_{n-1}))$; the theorem follows from this claim. To see this, observe that the set of of elements
\[
(T_{a_k} T_{a_{k-1}} \dots T_{a_2})*(T_{a_1} T_{a_1'}^{-1}) \in \overline{\Gamma_{D_n}}
\]
for $2 \le k \le n$ (where $a*b$ denotes the conjugation $aba^{-1}$) determines a set of point-pushes that generate $\pi_1(\Sigma(A_{n-1}))$.
\end{proof}

\section{Low multiplicity}\label{sec:sporadic}

In this final section we analyze the low-multiplicity cases not covered in \Cref{prop:generating_tree_mult_5}. Recall that our objective is \Cref{corollary:lowmult}:\\

\noindent \textbf{\Cref{corollary:lowmult}.} {\em 		Let $f:\C^2 \to \C$ be an isolated plane curve singularity of multiplicity $m \le  4$ such that $\Sigma(f)$ has genus at least $5$. If $f$ is not the $A_n$ or $D_n$ singularity, then $f$ admits a divide $\mathcal D$ with a complete subgraph $\mathcal C_0 \le \ld$ such that $\Sigma(\mathcal C_0)$ has genus $5$ and such that there is an $E$-arboreal spanning configuration on $\Sigma(\mathcal C_0)$ consisting of vanishing cycles. }\\

We analyze multiplicities $2,3,4$ separately. The case of multiplicity $2$ is trivial: every singularity of multiplicity $2$ is topologically equivalent to an $A_n$ singularity. For multiplicity $3,4$, more analysis is required, and in \Cref{subsection:tools}, we present some tools for working with divides. In \Cref{section:strategy}, we describe the general structure of the argument. We then employ this strategy to analyze multiplicity $3$ in \Cref{section:m3}, and treat multiplicity $4$ in \Cref{section:m4}. 

\subsection{Further tools}\label{subsection:tools}

In this section, we present three tools needed in the sequel to produce suitable divides and to manipulate collections of vanishing cycles.

\para{Admissible isotopies} As explained in \cite[Chapter 4]{ArnII} one can modify a divide $\calD$ associated with an isolated plane curve singularity $f$ without changing all the invariants associated with it (including the topology of the associated Milnor fiber and the geometric monodromy group). Such modifications are realized by homotopies of the divide such that all modifications of the embedded isotopy type of the divide are as in \Cref{figure:admissible_homotopy}. That is, a segment may cross a double point of the divide.

\begin{figure}[ht]
	\includegraphics[scale=1]{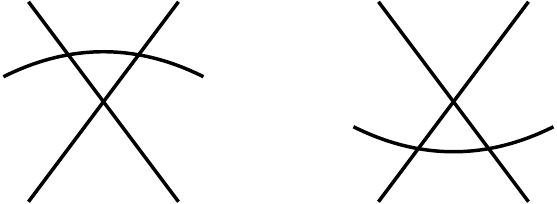}
	\caption{A change of the isotopy type of a divide by an admissible homotopy. Left and right figures are interchangeable in a given divide.}
	\label{figure:admissible_homotopy}
\end{figure}

\para{Chebyshev polynomials} There is a particularly nice way of producing divides for singularities of the type $f(x,y)=x^p-y^q$ using Chebyshev polynomials (we assume for this discussion that $\gcd(p,q) = 1$).  Without getting into details, (see \cite{Gus}), we assert that the function \begin{equation}\label{eq:chebyshev}\tilde{f}(x,y) =  \cos(p\cdot \cos^{-1}(x))-\cos(q\cdot \cos^{-1}(y))\end{equation} is a real morsification of $f(x,y)$ and that $\tilde{f}^{-1}(0)\cap [-1,1] \times [-1,1]$ is a divide for the singularity defined by $f$.

\begin{remark}\label{rem:reducible}
	In the above construction, if $\gcd(p,q)=1$ then the divide given by $\tilde f$ is irreducible and consists of an immersion of an interval $[-1,1]$ in $[-1,1] \times [-1,1]$. If $\gcd(p,q) \neq 1$ then we may encounter a divide with some immersed circles (recall \Cref{rem:extended_version}), but the resulting intersection diagram can still be taken as in \Cref{figure:int_diagram_56}.
	
\end{remark}

	It is a straightforward induction using the morsification given by the Chebyshev polynomials as \cref{eq:chebyshev} to check that the resulting divides have the form depicted in \Cref{figure:int_diagram_56}. In that picture we see a grid of $p-1$ by $q-1$ vertices corresponding to the vanishing cycles of the singularity $x^p+ y^q$. See also the cited article \cite{Gus} for a more general construction including other types of singularities.
	
	\begin{figure}[ht]
	\includegraphics[scale=1]{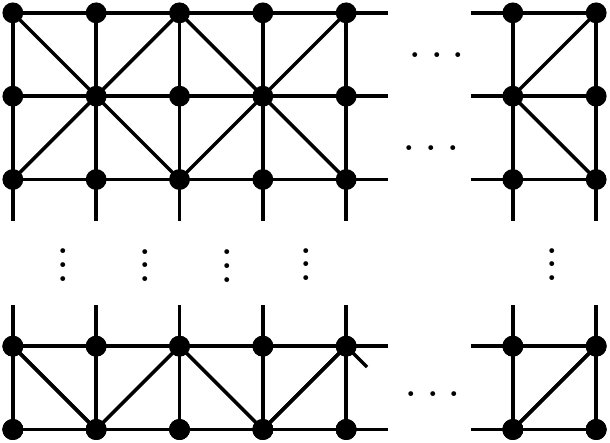}
	\caption{}
	\label{figure:int_diagram_56}
\end{figure}

\para{Change of basis: toggling triangles}
If $a$ and $b$ are vanishing cycles, the Dehn twists $T_a^{\pm 1}(b)$ are as well. If moreover $a$ and $b$ are elements of a distinguished basis (corresponding to vertices in an intersection graph $\Lambda_\mathcal D$), then by equipping $\Lambda_\mathcal D$ with the structure of a directed graph recording signed intersections, it is possible to read off the intersection pattern of $T_a^{\pm 1}(b)$ with the remaining elements of the distinguished basis. This has a simple encoding involving triangles in $\Lambda_\mathcal D$; we call this graphical procedure ``triangle toggling''. 

\begin{definition}[Oriented intersection graph, (in)coherent triangle]
	Let $\mathcal C = \{c_1, \dots, c_k\}$ be a distinguished set of vanishing cycles (i.e. a collection of pairwise non-isotopic curves such that $i(c_i, c_j) \le 1$ for all pairs); denote the associated intersection graph by $\Lambda_\mathcal D$. Endow each $c_i \in \mathcal C$ with an arbitrary orientation. The {\em oriented intersection graph} $\vec{\Lambda}_\mathcal C$ is the directed graph on the underlying undirected graph $\Lambda_\mathcal C$, where the edge between $c_i$ and $c_j$ is oriented towards $c_j$ if and only if the algebraic intersection $\pair{c_i, c_j} = 1$. 
	
	Suppose that $c_i, c_j, c_k$ form a triangle $T$ in $\Lambda_\mathcal C$. We say that $T$ is {\em coherent} if the directed subgraph forms a directed $3$-cycle, and is {\em incoherent} otherwise. 
\end{definition}

Suppose that there is a directed edge in $\vec{\Lambda}_\mathcal C$ from $c_i$ to $c_j$, i.e. that $\pair{c_i, c_j} = 1$. The lemma below says that when all triangles involving $c_i$ and $c_j$ are coherent, the effect of replacing $c_j$ with $T_{c_i}(c_j)$ is to ``toggle triangles'': if $c_k$ is adjacent to $c_i$, then remove the edge between $c_j$ and $c_k$ if it exists, and add in such an edge if it does not. 

\begin{lemma}[Triangle toggling]\label{lemma:toggle}
	Let $\vec{\Lambda}_{\mathcal C}$ be the oriented intersection graph associated to a configuration of curves $\mathcal C = \{c_1, \dots, c_k\}$.
	\begin{enumerate}
		\item\label{item:noedge} Suppose $i(c_i, c_k) = 1$ and $i(c_j, c_k) = 0$. Then $i(T_{c_i}(c_j), c_k) = 1$ and $\pair{T_{c_i}(c_j), c_k} =1$.
		\item\label{item:edge} Suppose $c_i, c_j, c_k$ determine a coherent triangle in $\vec{\Lambda}_{\mathcal C}$ with an edge directed from $c_i$ to $c_j$. Then $i(T_{c_i}(c_j), c_k) = 0$.
	\end{enumerate}
\end{lemma}

\begin{proof}
	Claim \eqref{item:noedge} simply records a basic fact concerning the algebra of Dehn twisting (c.f. \cite[Proposition 3.2, Proposition 6.3]{Farb}). Claim \eqref{item:edge} is similarly easy to prove by direct pictorial computation.
\end{proof}

The effect of toggling triangles is to ``change basis''. In particular, neither the topology of the subsurface spanned by the vertices in question, nor the geometric monodromy group, is affected by a toggle.

\begin{lemma}\label{lemma:nochange}
Let $\mathcal D$ be a divide, and let $\mathcal C \subset \ld$ be a subgraph. Let $\mathcal C'$ be the graph obtained from $\mathcal C$ by any sequence of triangle toggles applied to pairs of vertices both in $\mathcal C$. Then the subsurfaces $\Sigma(\mathcal C)$ and $\Sigma(\mathcal C')$ are isotopic, and there is an equality of subgroups of $\Mod(\Sigma(\mathcal D))$
\[
\pair{T_c \mid c \in \mathcal C} = \pair{T_c \mid c \in \mathcal C'}.
\]
\end{lemma}

Observe that it is always possible to orient any three $c_i, c_j, c_k$ such that the associated triangle is coherent, but it may not be possible to orient all $c_i \in \mathcal C$ such that all triangles in $\vec{\Lambda}_\mathcal C$ are simultaneously coherent. However, for intersection graphs arising from divides, this is the case.

\begin{lemma}\label{lemma:oriented}
	Let $\mathcal D$ be a divide with associated vanishing cycles $\mathcal C = \{c_1, \dots, c_k\}$ and intersection graph $\Lambda_\mathcal D$. Then there exists an orientation of the elements of $\mathcal C$ such that all triangles in the associated $\vec{\Lambda}_\mathcal D$ are coherent. 
\end{lemma}

\begin{proof}
	This is a reformulation of a basic fact concerning the structure of the intersection form associated to a Milnor fiber. Compare \cite[Theorem 4.1]{ArnII}.
\end{proof}

\begin{remark}[Unoriented intersection graphs suffice]\label{remark:unoriented}
The general theory of triangle toggling certainly requires the extra information encoded in the {\em oriented} intersection graph. However, our use of toggling in this paper will be simple enough that we will not need to pay attention to orientations, and consequently the toggling computations in \Cref{section:m3,section:m4} take place on the unoriented intersection graphs. We will only ever use the toggling procedure to remove triangles that appear in the original intersection graphs arising from divides, and these are necessarily coherently oriented by \Cref{lemma:oriented}. In \Cref{section:m4}, we indicate the toggling procedure with a colored directed arrow; if the arrow points from $a$ to $b$, this indicates that $b$ is replaced by $T_a(b)$. 
\end{remark}

\subsection{Outline of the argument}\label{section:strategy} We explain here the strategy employed in analyzing the remaining cases. 

\para{Step 1: Produce (partial) divides} The first step in the arguments below is to produce divides for the singularities under consideration. As in \Cref{prop:generating_tree_mult_5}, it will often suffice to find some universal portion of a divide applicable to a class of singularities. For multiplicity $4$, we are able to exploit the same strategy as in \Cref{prop:generating_tree_mult_5}, by deforming a divide with ordinary singularities and proceeding by a reasonable quantity of casework. For multiplicity $3$, we first partition the analysis by the number of branches of the singularity, and employ some more ad-hoc methods. 

\para{Step 2 : Identify $\boldmath \mathcal C_0$} Step $1$ produces a divide $\mathcal D$ for the singularities under study. The second step is to find a connected subgraph $\mathcal C_0 \subset \ld$ with $10$ vertices, corresponding to a subsurface $\Sigma(\mathcal C_0)$ of genus $5$. In ideal circumstances, $\mathcal C_0$ will already be $E$-arboreal. However, much of the time, we will require the additional Step 3 below to manipulate the vanishing cycles until the intersection graph becomes arboreal.

\para{Step 3: Toggle triangles} The final step of the argument is to apply the triangle toggling procedure to $\mathcal C_0$. The goal is to take $\mathcal C_0$ to a new subgraph $\mathcal C_0'$ that has the structure of an $E$-arboreal spanning configuration. The curves corresponding to the new vertices are obtained from the distinguished vanishing cycles by a sequence of Dehn twists, and so are themselves vanishing cycles.

\subsection{Multiplicity $3$}\label{section:m3} We divide this case in three subcases depending on the number of branches of the singularity. 

\para{One branch}
We assume first that the singularity is irreducible. In this case it is topologically equivalent to an irreducible singularity with the Puiseux pair $(3,q)$ with $q > 3$ and $\gcd(3,q)=1$, i.e. up to change of coordinates, it has the form $x^3 + y^q$.

We construct the divide $\mathcal D$ associated to the Chebyshev polynomial as in \Cref{figure:int_diagram_56}. If $q \geq 10$ then the $(1,2,6)$ tripod graph is a complete subgraph $\mathcal C_0 \subset \ld$: take $9$ vertices from one of the rows of the intersection diagram  and one extra suitable vertex in an adjacent row (see again \Cref{figure:int_diagram_56}). This provides the required $E$-arboreal spanning configuration.

It remains to analyze the singularities of the form $x^3+y^q$ with $q<10$ for which the Milnor fiber has genus at least $5$. There are two such singularities: those with Puiseux pairs $(3,7)$ and $ (3,8)$.

In each case, we construct the divide $\mathcal D$ associated to the Chebyshev polynomial as before, but now we modify the divide via an admissible homotopy (\Cref{subsection:tools}), leading to a new divide $\mathcal D'$. We then identify a subgraph $\mathcal C_0 \subset \Lambda_{\mathcal D'}$ with $10$ vertices.

We simply draw a divide for each of these singularities and, after admissible homotopies and toggling triangles, we find a complete subgraph of the corresponding intersection diagram which is $E$-arboreal of genus at least $5$. See \Cref{figure:divide37,figure:dual37} and \Cref{figure:divide38,figure:dual38} for each of the two cases. For the Puiseux pair $(3,8)$ it is possible to find $\mathcal C_0$ with the structure of the $(1,2,6)$ tripod graph (and hence to conclude as above), but for Puiseux pair $(3,7)$, more work is required. We use the triangle-toggling procedure of \Cref{subsection:tools} to modify $\mathcal C_0$ to a new graph $\mathcal C_0'$ with the structure of the $(1,2,6)$ tripod graph and conclude.
\begin{figure}[ht!]
		\includegraphics[scale=1]{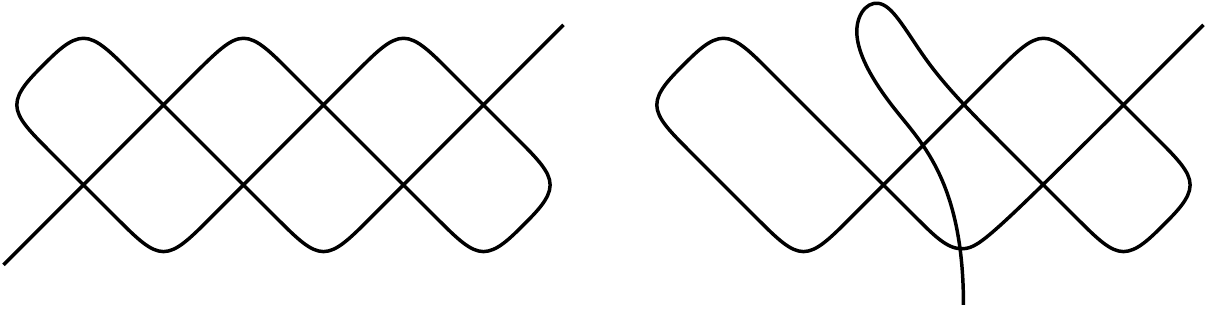}
	\caption{On the left, the standard divide for the singularity $x^3+y^7$. On the right the transformed divide after an admissible isotopy.}
	\label{figure:divide37}
\end{figure}

\begin{figure}[ht!]
	\includegraphics[scale=1]{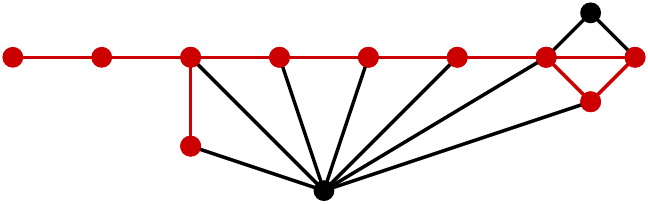}
	\caption{Here we see the associated dual graph to the deformed divide. The red subgraph becomes $E$-arboreal after toggling its only triangle.}
	\label{figure:dual37}
\end{figure}

\begin{figure}[ht!]
	\includegraphics[scale=1]{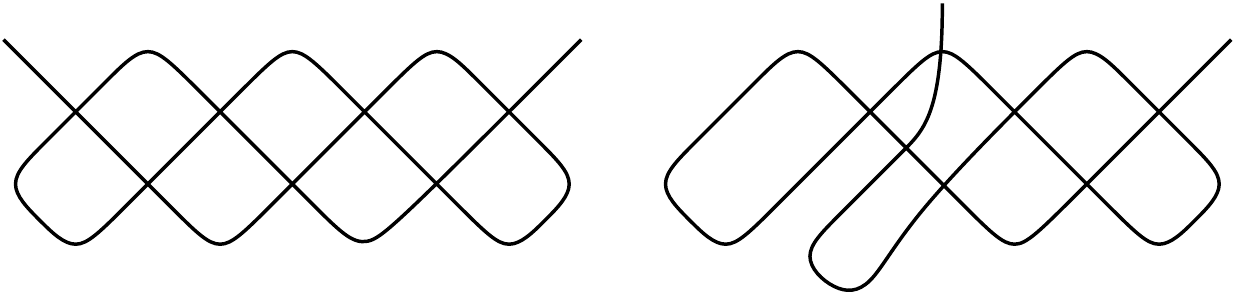}
	\caption{On the left, the standard divide for the singularity $x^3+y^8$. On the right the transformed divide after an admissible isotopy.}
	\label{figure:divide38}
\end{figure}

\begin{figure}[ht!]
	\includegraphics[scale=1]{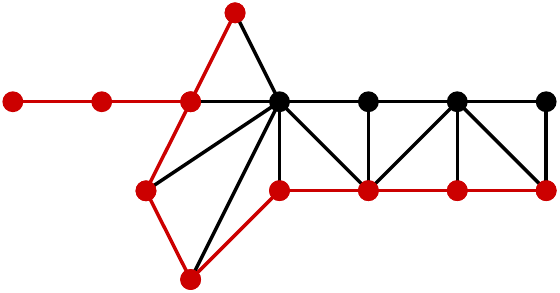}
	\caption{Here we see the associated dual graph to the deformed divide. In red, an $E$-arboreal spanning configuration is depicted.}
	\label{figure:dual38}
\end{figure}

\para{Two branches}  The following lemma completely classifies isolated plane curve singularities of multiplicity $3$ consisting of exactly $2$ branches.

\begin{lemma}\label{lem:class_mult3_2branch}
Let $f:\C^2 \to \C$ be an isolated plane curve singularity of multiplicity $3$ and assume that $f$ consists of exactly two branches. Then $f$ is topologically equivalent to a singularity consisting of a branch topologically equivalent to an $A_k$ singularity for $k$ even, and a branch which is smooth. The intersection multiplicity $\nu$ of the two branches is either $k+1$ or else $2d$ for some $d \in \N$ with $2 \leq 2d < k+1$. 

Up to a topological change of coordinates, any such $f$ has an expression of the form
\begin{equation}\label{equation:models}
f=(x+y^d)(x^2-y^{k+1})
\end{equation}
with $1 \le 2d \le k+2$. For $2d < k+2$, the intersection multiplicity is $\nu = 2d$, while for $2d = k+2$, it is $\nu = k+1$. 
\end{lemma}

\begin{proof}
The only singularities of multiplicity $2$ and one branch are the $A_{k}$ singularities $x^2+y^{k+1}$ for $k \in 2 \N$. So we find that the only singularities of multiplicity $3$ and two branches must have a branch topologically equivalent to an $A_{k}$ singularity and a smooth branch. In order to determine its topological type we must only specify $k$ and the intersection multiplicity between the two branches.

First observe that in order for $x^2+y^{k+1}$ to be irreducible, $k+1$ must be odd, hence the strict inequality in the hypothesis. 

Up to change of coordinates, we can assume that the singularity $f$ is of the form $f_1(x,y) \cdot (x^2- y^{k+1})$ where $f_1(x,y)$ is a polynomial of order $1$. This means that 
\[
f_1(x,y)= ax+by + \textit{higher order terms},
\]
and $a,b \in \C$ with at least one of $a$ or $b$ not zero. In order to compute the intersection multiplicity, we parametrize the branch $x^2-y^{k+1}$ via $t \mapsto (t^{k+1},t^2)$ and compute the order of $f_1(t^{k+1}, t^2)$. If $b \neq 0$ then $\nu = 2$ because no other term can cancel the term $bt^2$. If $b=0$, then necessarily $a \neq 0$. The term $a t^{k+1}$ cannot be cancelled by any higher-order terms of $f_1(x,y)$, and so $\nu \le k+1$. The inequality is strict if there is a monomial of the form $y^d = t^{2d}$ with $2d < k+1$, in which case $\nu = 2d$ for the minimal such $d$. 

It is now clear that the explicit expressions in the statement of the lemma have the required properties.
\end{proof}

We can now proceed with the Steps 1 - 3 of \Cref{section:strategy}, as applied to the models given in \cref{equation:models}. The case $d=1$ yields the $D_k$ singularity. Assume therefore that $d \geq 2$. In this case the Milnor fiber has $2$ boundary components and, following \cref{eq:milnor_total}, genus $g= k/2 + \nu -1$. Recall that $\nu \le k+1$ and that $k$ is even, so that for $k =2$, the condition $g \ge 5$ is impossible. For $k \ge 4$, all values of $\nu$ for $d \ge 2$ yield Milnor fibers of genus $g \ge 5$.

To produce divides, we observe that each of the branches can be morsified by \cref{eq:chebyshev}. After doing so, it can be directly checked that the divides are as in \cref{figure:mult3cases}. 

\begin{figure}[ht!]
	\includegraphics[scale=1]{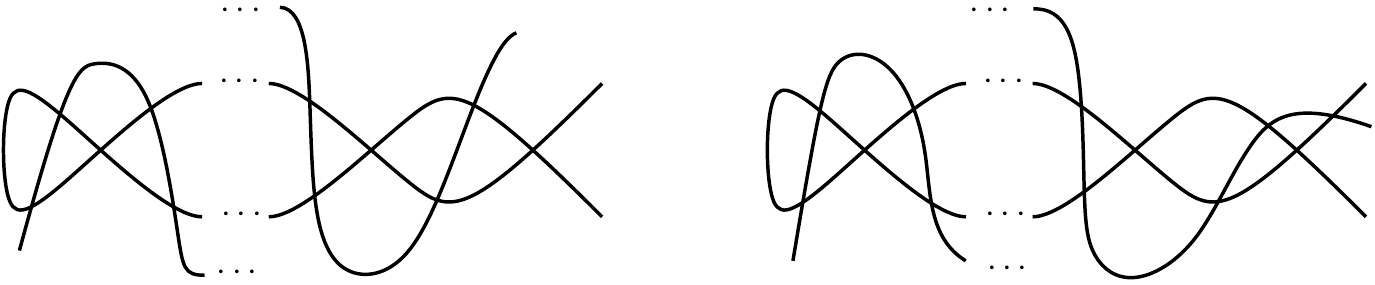}
	\caption{The divide of the singularity $(x+y^d)(x^2-y^{k+1})$. At left, we have depicted the maximal $d$ with $2d<k+1$, and on the right, we depict the maximal intersection multiplicity $d = (k+2)/2$. For smaller $d$, there are simply fewer ``oscillations'' of the smooth branch across the divide for $A_k$.}
	\label{figure:mult3cases}
\end{figure}

Every divide as in \Cref{figure:mult3cases} has a region as shown at left in \Cref{figure:mult32branches}. At right, we have applied an admissible isotopy, completing Step 1. In red, we see a complete subgraph of the intersection graph with two triangles (Step 2). After toggling (\Cref{lemma:toggle}) both triangles, the intersection graph becomes the tripod graph of type $(1,2,6)$, completing the argument. 

\begin{figure}[ht!]
	\includegraphics[scale=1]{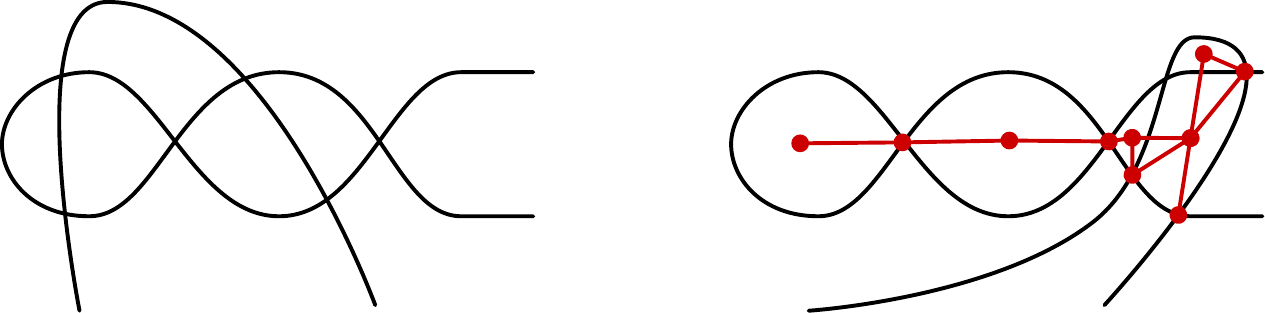}
	\caption{}
	\label{figure:mult32branches}
\end{figure}

\para{Three branches} The last case in multiplicity $3$ is when there are $3$ branches. In this case, all branches must be smooth so the topological type of the singularity is determined by the pairwise intersection multiplicities of the branches. We do a similar analysis to that of \Cref{lem:class_mult3_2branch} to classify the topological types of singularities that appear in this situation.

\begin{lemma}\label{lem:mult3branch3}
	Let $f:\C^2\to \C$ be an isolated plane curve singularity of multiplicity $3$ and assume that $f$ consists of $3$ branches. Then the branches $f_1,f_2$ and $f_3$ are smooth,  at least two of the intersection multiplicities coincide, and the remaining multiplicity is greater or equal than the other two. Moreover, all possible sets of natural numbers $\{\nu_{12},\nu_{13},\nu_{23}\}$ that satisfy the previous conditions can be realized.
\end{lemma}

\begin{proof}
	That each branch is smooth follows from the restriction that $f$ has multiplicity $3$ and that each branch has multiplicity at least $1$.
	
	We now analyze the possible intersection multiplicities between the branches. Take the branches $f_2$ and $f_3$. If their intersection multiplicity is $\nu_{23} \geq 1$, then, up to a change of coordinates, we claim that
	 \[
	 f=f_1(x,y) \cdot (x^2-y^{2\nu_{23}}),
	 \]
	  where $f_1(x,y)$ is a polynomial of order $1$ and $\nu_{23}$ can be any natural number. To see this, we observe that the topological type of a singularity consisting of two smooth branches is determined by the intersection multiplicity of these two branches. 
	  	
	Similar to the argument of \Cref{lem:class_mult3_2branch}, we now analyze the possible intersection multiplicities of $f_2=x-y^{\nu_{23}}$ and $f_3=x+y^{\nu_{23}}$ with $f_1(x,y)$. 
	
	We observe that $f_2= x-y^{\nu_{23}}$  has a parametrization $t \mapsto (t^{\nu_{23}}, t)$ and $f_3=x+y^{\nu_{23}}$ has the parametrization $t \mapsto (-t^{\nu_{23}}, t)$. We write $f_1(x,y)=ax+by^k+\textit{higher order terms}$ for some $k \ge 1$ (since $f_1$ is irreducible, necessarily some monomial of the form $b y^k$ must appear). 
	
	We subdivide the analysis in terms of the possible values for $k$. Throughout, we identify the lowest-order terms in the expressions $f_1(\pm t^{\nu_{23}}, t)$ which determine intersection multiplicities $\nu_{12}, \nu_{13}$. 
	
	\para{Case 1} If $k \ne \nu_{23}$, then $\nu_{12} = \nu_{13} = k$ (if $k < \nu_{23}$) or else $\nu_{12} = \nu_{13} = \nu_{23}$ (if $k > \nu_{23})$. In either case, the trio of multiplicities $\nu_{12}, \nu_{13}, \nu_{23}$ behaves as claimed.
		
	\para{Case 2} If $k=\nu_{23}$, then the coefficient on $t^{\nu_{23}}$ in the substitution $f_1(t^{\nu_{23}}, t)$ (computing the intersection multiplicity $\nu_{12}$) is $a+b$, and is $b-a$ in the substitution $f_1(-t^{\nu_{23}}, t)$ that computes $\nu_{13}$. Thus at least one of these coefficients is nonvanishing, at least one of $\nu_{12}$ or $\nu_{13}$ coincides with $\nu_{23}$, and the remaining one is at least as large.

	It remains to show that any three natural numbers $\nu_{12} = \nu_{13} \le \nu_{23}$ can be realized as multiplicities.
	
	\para{Case (a)} If $\nu_{12}=\nu_{13}=\nu_{23}$ then the singularity \[x^3-y^{3\nu_{12}}\] realizes the multiplicities.

	\para{Case (b)} For $\nu_{12}=\nu_{13} < \nu_{23}$, use \[(x-y^{\nu_{12}})\cdot(x^2-y^{2\nu_{23}}).\]

\end{proof}

Having classified the singularities in this regime, we apply Steps 1-3 of \Cref{section:strategy}. We begin with Step 1, constructing suitable divides. For the case $(a)$, (that is, the singularities of the form $x^3-y^{3m}$), we appeal to the construction given by the corresponding Chebyshev polynomial (recall \Cref{rem:reducible}). The genus of the Milnor fibers of these singularities is $g=3m-2$, and so $g \geq 5$ if and only if $m \geq 3$. Given the intersection diagram as in \Cref{figure:int_diagram_56} it is clear that we only need to check the limit case $x^3-y^9$. In \Cref{figure:case39} we see the intersection diagram (on top) where the graph $\mathcal C_0$ of Step 2 is depicted in red, completing Step 2. After toggling the only triangle it contains (using \Cref{lemma:toggle}), we obtain a tripod graph of type $(1,2,6)$, completing Step 3.

\begin{figure}[ht!]
	\includegraphics[scale=1]{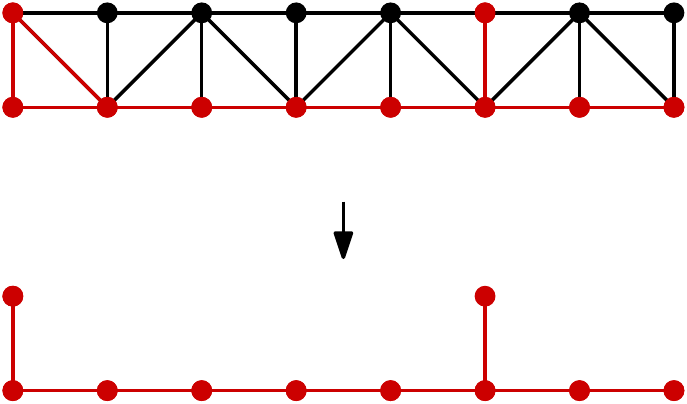}
	\caption{}
	\label{figure:case39}
\end{figure}

 The divides corresponding to singularities as in Case $(b)$ of the proof of \Cref{lem:mult3branch3} (of the form $(x-y^{\nu_{12}})\cdot(x^2-y^{2\nu_{23}})$) are constructed in a very similar way as in the case of two branches (recall \Cref{figure:mult3cases}). First we analyze which of these singularities has genus greater or equal than $5$. Using \cref{eq:milnor_total} we find $g = \nu_{12}+\nu_{13}+\nu_{23}-2$. So, assuming $\nu_{23}>\nu_{12}=\nu_{13}$, we find that  $g \geq 5$ if and only if $2\nu_{12}+\nu_{23} \geq 7$. Since $\nu_{23} > \nu_{12} \geq 1$, we find that, either $\nu_{12}=\nu_{13}=1$ and $\nu_{23} \geq 5$ or else $\nu_{12}=\nu_{13} \ge 2$ and $\nu_{23} \geq 3$. The first possibility is a singularity of type $D_k$ and so is excluded from consideration. In the latter case we apply the Chebyshev morsification of \cref{eq:chebyshev} to each of the branches and place them in generic position to produce a divide (Step 1). In \Cref{figure:mult3branch3cases} we see a portion of an arbitrary such divide after applying a suitable admissible isotopy. The complete subgraph $\mathcal C_0$ of Step 2 is shown in red. After toggling (\Cref{lemma:toggle}) the only triangle, we obtain a tripod graph of type $(1,2,6)$, completing Step 3.
 
 \begin{figure}[ht!]
 	\includegraphics[scale=1]{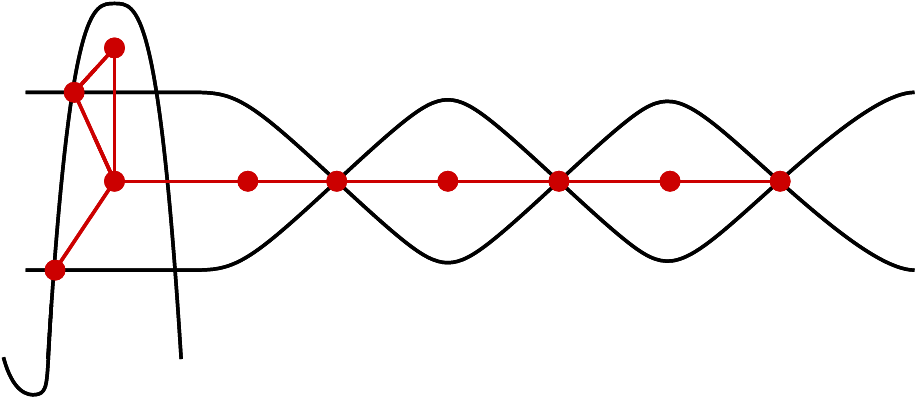}
 	\caption{After using the Chebyshev polynomial construction of the divide on each branch and after a suitable admissible isotopy, this is the divide corresponding to the singularity $(x-y^2)(x^2-y^6)$.}
 	\label{figure:mult3branch3cases}
 \end{figure}

\subsection{Multiplicity $4$}\label{section:m4}
Throughout this section, let $f: \C^2 \to \C$ be an isolated plane curve singularity of multiplicity $4$. Our first objective, corresponding to Step 1 of the outline given in \Cref{section:strategy}, is to analyze the structure of a divide under our standing assumption that the Milnor fiber $\Sigma_f$ has genus at least $5$. By \Cref{lem:divide_with_ordinary}, $f$ admits a divide $\mathcal D$ with an ordinary singularity equivalent to $x^4 + y^4$ at the origin. Arguing as in the proof of \Cref{prop:generating_tree_mult_5}, $\mathcal D$ can be deformed so as to have the structure of four lines in general position, enclosing three bounded regions. Taken as a divide unto itself, the corresponding subsurface has genus $3$ and $4$ boundary components. Thus, in order for $\Sigma_f$ to have genus at least $5$, there must be additional vertices in the intersection graph, corresponding to ordinary double points and enclosed regions in $\mathcal D$. In \Cref{lemma:twomore}, we show that in fact there must be at least two more bounded regions in $\mathcal D$.

\begin{lemma}\label{lemma:twomore}
	Let $\mathcal D$ be a divide associated to an isolated real plane curve singularity $f: \C^2 \to \C$. The genus $g$ of the Milnor fiber $\Sigma_f$ coincides with the number of bounded regions  of $\R^2$ enclosed by $\mathcal D$.
\end{lemma}

\begin{proof}
	Let $\mu$ be the Milnor number of $f$, which coincides with the first Betti number of the Milnor fiber. Let $b$ be the number of branches of $f$, which coincides with the number of boundary components of its Milnor fiber. Let $\delta$ be the {\em delta invariant} associated with $f$, i.e. the number of double points appearing in any generic divide $\calD$ associated with $f$. Finally, let $r$ denote the number of bounded regions enclosed by $\calD$.
	
	We have the following three formulas (the first simply expresses $\mu$ as the first Betti number of the fiber, the second computes $\mu$ from the divide and the third is well-known;  see \cite{Mil}).
	\begin{equation}\label{eq:milnor_betti}
	\frac{\mu  - b + 1}{2} = g
	\end{equation}
	\begin{equation}\label{eq:milnor_divide}
	\mu = r + \delta
	\end{equation}
	\begin{equation}\label{eq:delta_inv}
	\mu = 2\delta - b + 1.
	\end{equation}
	From \cref{eq:delta_inv} and \cref{eq:milnor_betti}, we find $g = \delta - b + 1 = \mu - \delta$, and by \cref{eq:milnor_divide}, $\mu-\delta = r$ as required. 	
\end{proof}

We now see that up to admissible homotopy, there are only five cases to consider. 

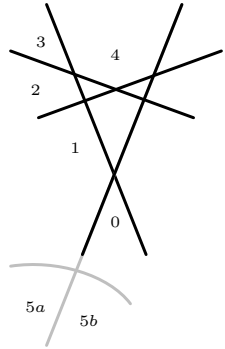
\begin{figure}[ht!]
	\labellist
	\tiny
	\pinlabel $5a$ at 10 15
	\pinlabel $5b$ at 30 10
	\pinlabel $0$ at 40 47
	\pinlabel $1$ at 25 75
	\pinlabel $2$ at 10 97
	\pinlabel $3$ at 12 115
	\pinlabel $4$ at 40 110
	\endlabellist
	\includegraphics[scale=1]{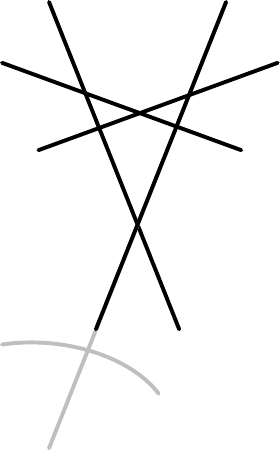}
	\caption{The possible arrangements of the two extra regions for multiplicity $4$. Every such singularity has a divide including the four lines in general position and such that the region labeled $0$ is enclosed. Additionally, at least one of the regions labeled $1$ through $5b$ is enclosed.}
	\label{figure:mult4cases}
\end{figure}

\begin{lemma}\label{lemma:mult4cases}
	Let $f: \C^2 \to \C$ be an isolated plane curve singularity of multiplicity $4$ such that $g(\Sigma_f) \ge 5$. Then $f$ admits a divide $\mathcal D$ containing a region of one of the forms $1$--$5$ as indicated in \Cref{figure:mult4cases}. 
\end{lemma}

\begin{proof}
	By \Cref{lem:properties_divides_singularities}.\ref{prop:i}, the two additional enclosed regions guaranteed by \Cref{lemma:twomore} must be placed in the divide in such a way that the intersection diagram $\Lambda_\mathcal D$ is connected. There are eight local regions adjacent to the three regions enclosed by the four lines in general position; see \Cref{figure:mult4cases}. Thus at least one of these regions must be enclosed. There are two basic regimes: either exactly one such region is enclosed, or else at least two are. 
	
	\begin{figure}[ht!]
		\includegraphics[scale=1]{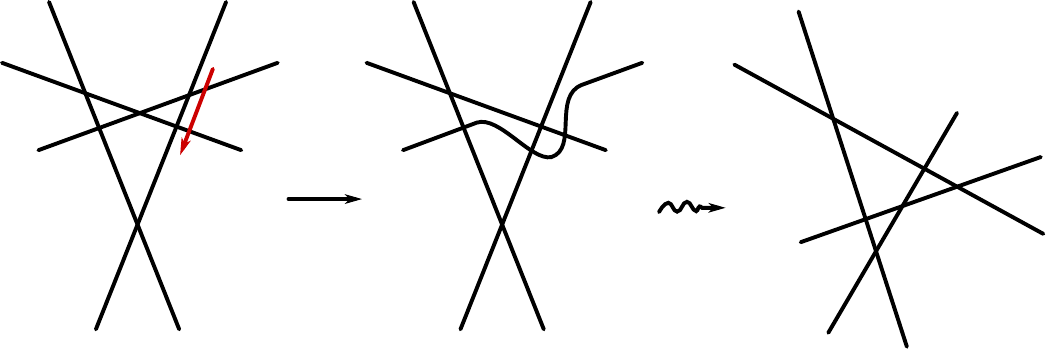}
		\caption{An admissible homotopy effects a rotation of the position of the four lines relative to the adjacent regions.}
		\label{figure:rotation}
	\end{figure}

	In either case, we claim that two configurations are related by admissible homotopies if the positions of the adjacent regions are equivalent up to a rotation of the configuration of four lines. This follows directly from \Cref{figure:rotation}, which shows that a single admissible homotopy of one of the four lines amounts to a rotation through three of the eight adjacent positions. Thus, a configuration with two adjacent enclosed regions is determined entirely by the number of empty regions separating the two enclosed regions; there are four such possibilities. Additionally, any two configurations with exactly one adjacent enclosed region are related by a sequence of admissible homotopies, leading to a fifth case to consider. There are two variants $5a$ and $5b$, depending on whether the fifth adjacent region shares an edge or only a vertex with the fourth.
\end{proof}

%\para{Analysis of cases: outline} To conclude the argument in the case of multiplicity $4$, we analyze each of the cases $1-5$ enumerated in Lemma \ref{lemma:mult4cases}. The basic strategy is the same in each case. If necessary, we first make some further observations that supply us with additional vanishing cycles. We arrive at a divide with $11$ or $12$ vanishing cycles supporting a surface of genus $5$. We then judiciously choose one or more vanishing cycles to ignore, leading to a surface of genus $5$ with one boundary component. The intersection graph for this will not be a tree and so does not fall under the purview of Theorem \ref{theorem:Egens}, but by performing a change-of-basis (the ``triangle toggling'' procedure of Lemma \ref{lemma:toggle}), we can replace our configuration of vanishing cycles with a new set whose intersection graph forms a tree and satisfies the hypotheses of Theorem \ref{theorem:Egens}. 

For some of the cases below, we will require a slight additional argument to complete Step 1; we will make some further observations that supply us with additional vanishing cycles. Once this is taken care of, the remainder of Steps 2 and 3 proceed as in the multiplicity $3$ case.

\para{Case 1}
\begin{figure}[ht!]
	\labellist
	\small
	\pinlabel $(a)$ at 5 120
	\pinlabel $(b)$ at 150 120
	\pinlabel $(c)$ at -8 5
	\pinlabel $(d)$ at 82 5
	\pinlabel $(e)$ at 172 5
	\tiny
	\pinlabel $v$ at 25 80
	\pinlabel $w$ at 33 23
	\endlabellist
	\includegraphics[scale=1]{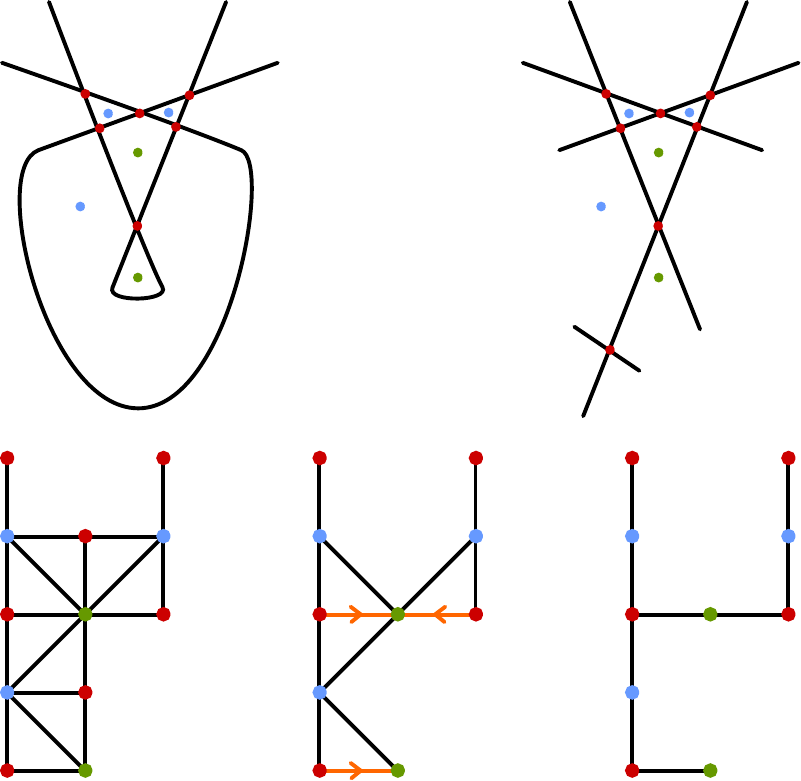}
	\caption{Case 1}
	\label{figure:case1}
\end{figure}
We refer to \Cref{figure:case1} throughout. The divides labeled $(a)$ and $(b)$ in the top half of the figure indicate two distinct possibilities. Either both of the extra enclosed regions has no additional edges, as shown in $(a)$, or else one of the regions has some additional vertex on its boundary. Without loss of generality (c.f. \Cref{figure:rotation}), we can assume that in this case the figure appears as in $(b)$. 

Assume first that the a portion of the divide appears as in \Cref{figure:case1} (a). We divide this in two more subcases.

\para{Subcase (a)'} The whole divide of the singularity is as in $(a)$ - there are no extra enclosed regions or double points apart from those in $(a)$. Then we observe that the singularity consists of two branches, each of them topologically equivalent to $x^2-y^3$ and with intersection multiplicity between the branches equal to $4$ (these are the number of intersection points between the two nodes in the picture). We find that in this case the singularity is topologically equivalent to $(x^2-y^3)(y^2-x^3)$. Using \cref{eq:chebyshev} we can produce a divide for each of the branches and find that the divide of \Cref{figure:dividedoublecusp} is a divide for the singularity and, furthermore, this divide falls under the purview of Case $2$ in the scheme of \Cref{lemma:mult4cases}.

\begin{figure}[ht!]
	\includegraphics[scale=1]{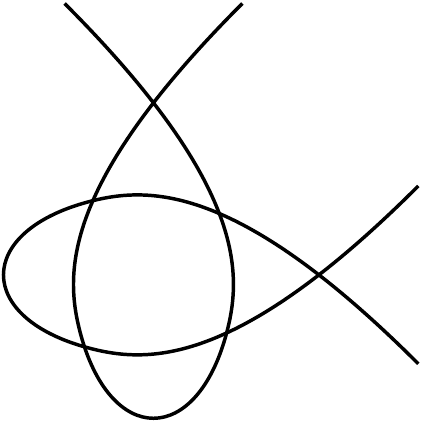}
	\caption{A divide for the singularity $(x^2-y^3)(y^2-x^3)$.}
	\label{figure:dividedoublecusp}
\end{figure}

\para{Subcase (a)''} Suppose that $(a)$ is not the whole divide of the singularity. Then there is some other region adjacent to the divide depicted in $(a)$ that is enclosed by the divide of the singularity. This case is then covered by one of the other cases.

Assume now the divide appears as in \Cref{figure:case1}.(b). The intersection diagram is shown in $(c)$, along with the distinguished vertices $v$ and $w$. In $(d)$, we show the portion of the intersection diagram remaining after ignoring the vertices $v$ and $w$; this is the subgraph $\mathcal C_0$ of Step 2. The highlighted edges in $(d)$ indicate the toggling moves to be performed for Step 3 (c.f. \Cref{remark:unoriented}). The result after toggling is shown in $(e)$. This is the $(2,3,4)$ tripod graph, which determines an $E$-admissible spanning configuration by \Cref{lemma:tripods} and completes Step 3 in this case.

\para{Case 2}
\begin{figure}[ht!]
	\labellist
	\small
	\pinlabel $(a)$ at 15 5
	\pinlabel $(b)$ at 105 5
	\pinlabel $(c)$ at 195 5
	\pinlabel $(d)$ at 285 5
	\tiny
	\pinlabel $v$ at 130 55
	\endlabellist
	\includegraphics[scale=1]{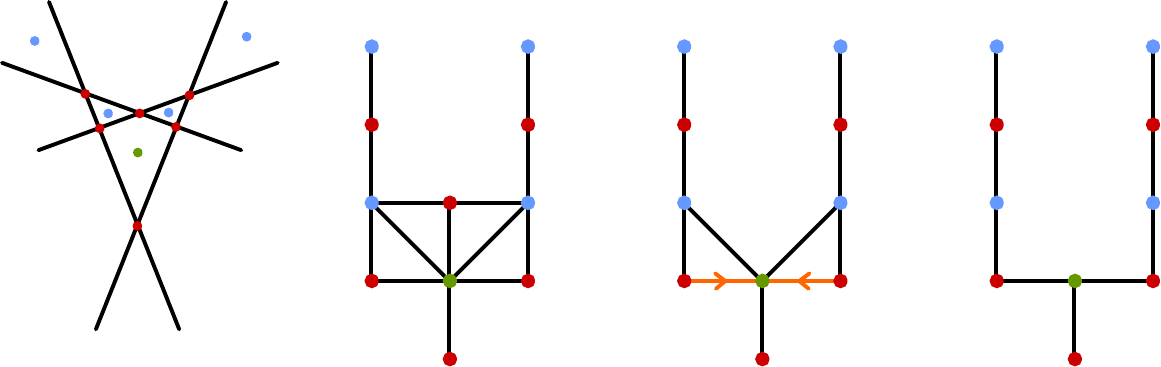}
	\caption{Case 2}
	\label{figure:case2}
\end{figure}
We refer to \Cref{figure:case2} throughout. For this case, it is simpler to place the bounded regions as shown in $(a)$; we are free to do this by the argument of \Cref{figure:rotation}. The intersection diagram is shown in $(b)$, along with the distinguished vertex $v$. In $(c)$, we show the portion of the intersection diagram remaining after ignoring $v$. The highlighted edges in $(c)$ indicate the toggling moves to be performed. The result is shown in $(e)$. This is the $(1,4,4)$ tripod graph. By \Cref{lemma:tripods}, this completes Step 3.

\para{Case 3}
\begin{figure}[ht!]
	\labellist
	\endlabellist
	\includegraphics[scale=1]{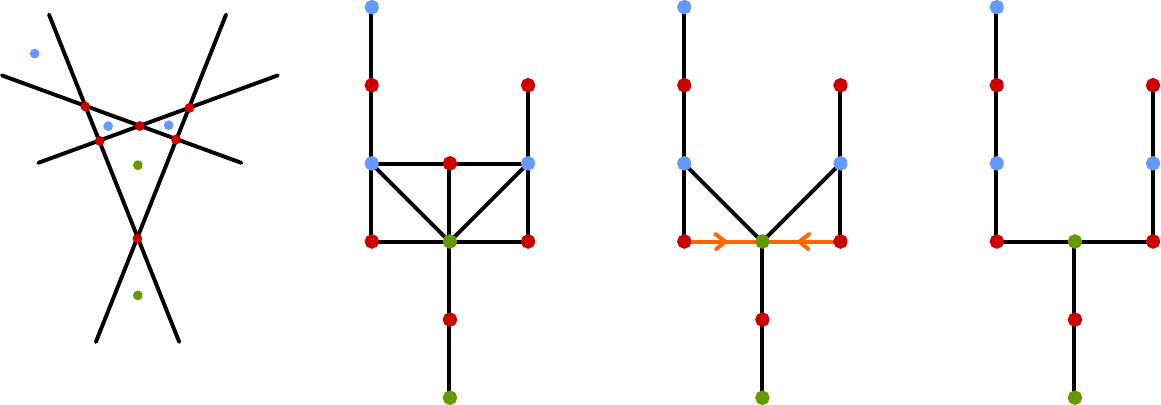}
	\caption{Case 3}
	\label{figure:case3}
\end{figure}
Refer to \Cref{figure:case3}. The argument proceeds exactly as in the previous two cases. 

\para{Case 4}
\begin{figure}[ht!]
	\labellist
	\endlabellist
	\includegraphics[scale=1]{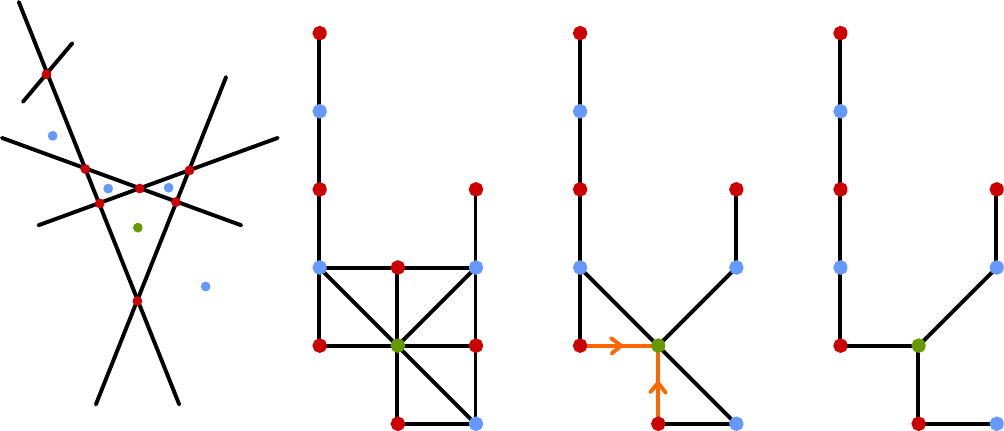}
	\caption{Case 4}
	\label{figure:case4}
\end{figure}
We refer to \Cref{figure:case4}. As in Case 2, we have chosen a different (but rotationally-equivalent) set of enclosed regions as compared to \Cref{lemma:mult4cases}. As we have depicted in the divide, there must be some additional edge crossing one of the two additional enclosed regions (otherwise, the divide would contain an immersed circle, in which case \Cref{lemma:twomore} does not apply, and the associated surface would only have genus $4$). The remainder of the argument proceeds as in the previous cases, terminating with a tripod graph of type $(2,2,5)$ which also falls under the purview of \Cref{lemma:tripods} and hence completes Step 3.

\para{Case 5}
\begin{figure}[ht!]
	\labellist
	\endlabellist
	\includegraphics[scale=1]{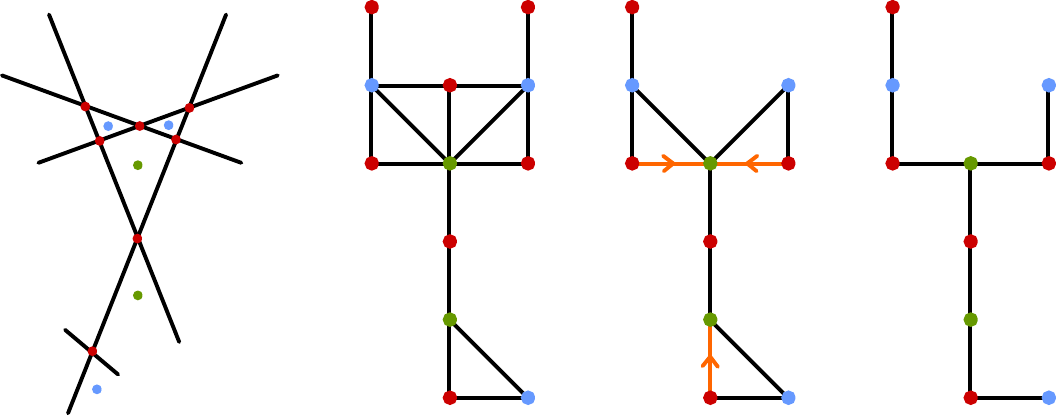}
	\caption{Case 5}
	\label{figure:case5}
\end{figure}
Finally we analyze Case 5 in \Cref{figure:case5}. There we have depicted the case labeled as 5b in \Cref{figure:mult4cases}. The analysis for case 5a is identical except that the bottom three vertices no longer form a complete triangle and there is one fewer toggling move required.

	\bibliographystyle{alpha}
	\bibliography{bibliography}

\begin{thebibliography}{AGZV12}

\bibitem[A'C73]{Nor3}
N.~A'Campo.
\newblock Sur la monodromie des singularit\'{e}s isol\'{e}es d'hypersurfaces
  complexes.
\newblock {\em Ann. Inst. Fourier (Grenoble)}, 23(2):1--2, 1973.

\bibitem[A'C75a]{NorGI}
N.~A'Campo.
\newblock Le groupe de monodromie du d\'{e}ploiement des singularit\'{e}s
  isol\'{e}es de courbes planes. {I}.
\newblock {\em Math. Ann.}, 213:1--32, 1975.

\bibitem[A'C75b]{NorGII}
N.~A'Campo.
\newblock Le groupe de monodromie du d\'{e}ploiement des singularit\'{e}s
  isol\'{e}es de courbes planes. {II}.
\newblock In {\em Proceedings of the {I}nternational {C}ongress of
  {M}athematicians ({V}ancouver, {B}. {C}., 1974), {V}ol. 1}, pages 395--404,
  1975.

\bibitem[A'C99]{Nor2}
N.~A'Campo.
\newblock Real deformations and complex topology of plane curve singularities.
\newblock {\em Ann. Fac. Sci. Toulouse Math. (6)}, 8(1):5--23, 1999.

\bibitem[A'C01]{Nor1}
N.~A'Campo.
\newblock Quadratic vanishing cycles, reduction curves and reduction of the
  monodromy group of plane curve singularities.
\newblock {\em Tohoku Math. J. (2)}, 53(4):533--552, 2001.

\bibitem[AGZV12]{ArnII}
V.~I. Arnold, S.~M. Gusein-Zade, and A.~N. Varchenko.
\newblock {\em Singularities of differentiable maps. {V}olume 2}.
\newblock Modern Birkh\"{a}user Classics. Birkh\"{a}user/Springer, New York,
  2012.
\newblock Monodromy and asymptotics of integrals, Translated from the Russian
  by Hugh Porteous and revised by the authors and James Montaldi, Reprint of
  the 1988 translation.

\bibitem[BK86]{Bri}
E.~Brieskorn and H.~Kn\"{o}rrer.
\newblock {\em Plane algebraic curves}.
\newblock Modern Birkh\"{a}user Classics. Birkh\"{a}user/Springer Basel AG,
  Basel, 1986.
\newblock Translated from the German original by John Stillwell, [2012] reprint
  of the 1986 edition.

\bibitem[Cas15]{Cast}
R.~Castellini.
\newblock {\em {The topology of A'Campo deformations of singularities: an
  approach through the lotus.}}
\newblock Theses, {Universit{\'e} de Lille 1}, September 2015.

\bibitem[CS20a]{strata3}
A.~{Calderon} and N.~{Salter}.
\newblock {Framed mapping class groups and the monodromy of strata of Abelian
  differentials}.
\newblock {\em arXiv e-prints}, page arXiv:2002.02472, Feb 2020.

\bibitem[CS20b]{framedMCG}
A.~{Calderon} and N.~{Salter}.
\newblock {Relative homological representations of framed mapping class
  groups}.
\newblock {\em arXiv e-prints}, page arXiv:2002.02471, Feb 2020.

\bibitem[dJvS98]{Jong}
T.~de~Jong and D.~van Straten.
\newblock Deformation theory of sandwiched singularities.
\newblock {\em Duke Math. J.}, 95(3):451--522, 1998.

\bibitem[FM12]{Farb}
B.~Farb and D.~Margalit.
\newblock {\em A primer on mapping class groups}, volume~49 of {\em Princeton
  Mathematical Series}.
\newblock Princeton University Press, Princeton, NJ, 2012.

\bibitem[Gab74]{Gab}
A.~M. Gabri\`elov.
\newblock Bifurcations, {D}ynkin diagrams and the modality of isolated
  singularities.
\newblock {\em Funkcional. Anal. i Prilo\v{z}en.}, 8(2):7--12, 1974.

\bibitem[GZ74]{Gus}
S.~M. Guse\u{\i}n-Zade.
\newblock Intersection matrices for certain singularities of functions of two
  variables.
\newblock {\em Funkcional. Anal. i Prilo\v{z}en.}, 8(1):11--15, 1974.

\bibitem[KS97]{Kha}
B.~Khanedani and T.~Suwa.
\newblock First variation of holomorphic forms and some applications.
\newblock {\em Hokkaido Math. J.}, 26(2):323--335, 02 1997.

\bibitem[Lab97]{Lab}
C.~Labruere.
\newblock Generalized braid groups and mapping class groups.
\newblock {\em J. Knot Theory Ramifications}, 6(5):715--726, 1997.

\bibitem[Mil68]{Mil}
J.~Milnor.
\newblock {\em Singular points of complex hypersurfaces}.
\newblock Annals of Mathematics Studies, No. 61. Princeton University Press,
  Princeton, N.J.; University of Tokyo Press, Tokyo, 1968.

\bibitem[PV96]{PV}
B.~Perron and J.~P. Vannier.
\newblock Groupe de monodromie g\'{e}om\'{e}trique des singularit\'{e}s
  simples.
\newblock {\em Math. Ann.}, 306(2):231--245, 1996.

\bibitem[RW13]{RW}
O.~Randal-Williams.
\newblock {Homology of the moduli spaces and mapping class groups of framed,
  $r$-Spin and Pin surfaces}.
\newblock {\em J. Topology}, 7(1):155--186, 2013.

\bibitem[Sal19]{NickToric}
N.~Salter.
\newblock Monodromy and vanishing cycles in toric surfaces.
\newblock {\em Invent. Math.}, 216(1):153--213, 2019.

\bibitem[Waj99]{Waj2}
B.~Wajnryb.
\newblock Artin groups and geometric monodromy.
\newblock {\em Invent. Math.}, 138(3):563--571, 1999.

\bibitem[Waj80]{wajnryb}
B.~Wajnryb.
\newblock On the monodromy group of plane curve singularities.
\newblock {\em Math. Ann.}, 246(2):141--154, 1979/80.

\end{thebibliography}
	
\end{document}